\documentclass[a4paper,11pt]{amsart}
\usepackage[utf8]{inputenc}
\usepackage{amssymb,amsmath,amsthm,amstext,amsfonts,color}
\usepackage[colorlinks=true,citecolor=black,linkcolor=black]{hyperref}
\usepackage[UKenglish]{babel}
\usepackage{enumitem}
\newtheorem{thmmain}{Theorem}

\newtheorem{thm}{Theorem}
\numberwithin{thm}{section}
\newtheorem{lem}[thm]{Lemma}

\newtheorem{prop}[thm]{Proposition}
\theoremstyle{definition}
\newtheorem{defin}[thm]{Definition}
\theoremstyle{remark}
\newtheorem*{rem}{Remark}

\newcommand{\abs}[1]{\left|#1\right|} 
\newcommand{\dist}[2]{\operatorname{d}(#1,#2)} 
\newcommand{\id}{\operatorname{id}} 
\newcommand{\fl}[1]{\phi^{#1}} 
\newcommand{\flow}[1]{\fl{#1}} 
\newcommand{\ball}[2]{\operatorname{B}({#1},{#2})} 
\newcommand{\cL}{\mathcal{L}} 
\newcommand{\cM}{\mathcal{M}} 
\newcommand{\cA}{\mathcal{A}} 
\newcommand{\cC}{\mathcal{C}} 
\newcommand{\bC}{\mathbb{C}} 
\newcommand{\cD}{\mathcal{D}} 
\newcommand{\bE}{\mathbb{E}} 
\newcommand{\cK}{\mathcal{K}} 
\newcommand{\bN}{\mathbb{N}} 
\newcommand{\bZ}{\mathbb{Z}} 
\newcommand{\cP}{\mathcal{P}} 
\newcommand{\cQ}{\mathcal{Q}} 
\newcommand{\cR}{\mathcal{R}} 
\newcommand{\norm}[1]{\left\|#1\right\|} 
\newcommand{\bR}{\mathbb{R}} 
\newcommand{\supp}{\operatorname{Supp}} 
\newcommand{\one}{\mathbf{1}} 
\newcommand{\VecA}{a}
\newcommand{\VecB}{b}
\newcommand{\A}{\mathbf{A}} 
\newcommand{\C}{\mathbf{C}} 
\newcommand{\D}{\mathbf{D}} 
\newcommand{\M}{\mathbf{M}} 
\newcommand{\Int}{\operatorname{Int}} 

\newcommand{\plus}{\raisebox{.11em}{$\scriptscriptstyle{+}$\!}}

\newcommand{\cexpand}{C_1} 
\newcommand{\cdistor}{C_2} 
\newcommand{\ceta}{C_3}
\newcommand{\cBoundTau}{C_{4}} 
\newcommand{\ctaup}{C_5}
\newcommand{\cdistorn}{C_6} 
\newcommand{\cbalance}{C_7} 
\newcommand{\cly}{C_{8}} 
\newcommand{\ctrans}{C_9}
\newcommand{\cPartUnityA}{C_{10}} 
\newcommand{\cPartUnityB}{C_{11}} 
\newcommand{\cRegK}{C_{12}} 
\newcommand{\cRegTheta}{C_{13}} 
\newcommand{\cOther}{C_{14}} 
\newcommand{\cCancel}{C_{15}}
\newcommand{\cdim}{d_{1}} 
\newcommand{\kZero}{K_{1}}
\newcommand{\kFill}{K_{2}}
\newcommand{\kInter}{K_{5}}
\newcommand{\kFour}{K_{4}}
\newcommand{\kThree}{K_{3}}

\title{Open Sets of Exponentially Mixing Anosov~Flows}
\thanks{With pleasure we thank Matias Delgadino, Stefano Luzzatto, Ian Melbourne, Masato Tsujii and Sina T\"ureli for stimulating discussions. We also thank Viviane Baladi, Fran\c{c}ois Ledrappier and the anonymous referee for highlighting an issue in a previous version of this paper. We are grateful to the ESI (Vienna) for hospitality during the  event ``Mixing Flows and Averaging Methods'' where this work was initiated. OB was partially supported by CNRS. KW was partially supported by DFG (CRC/TRR 191).}
\author{Oliver Butterley and Khadim War}
\address{(Oliver Butterley) Abdus Salam International Centre for Theoretical Physics -- Strada Costiera, 11 -- Trieste -- 34151 -- Italy.}
\email{oliver.butterley@ictp.it}
\address{(Khadim War) Faculty of Mathematics --  Ruhr-Universit\"at Bochum -- Universit{\"a}tsstra{\ss}e 150 -- 44801 Bochum -- Germany.}
\email{khadim.war@rub.de}
\date{\today}

\begin{document}

\begin{abstract}
We prove that an Anosov flow with \(\cC^{1}\) stable bundle mixes exponentially whenever the stable and unstable bundles are not jointly integrable.
This allows us to show that if a flow is sufficiently close to a volume-preserving Anosov flow and $\dim \bE_s = 1$, $\dim \bE_u \geq 2$ then the flow mixes exponentially whenever the stable and unstable bundles are not jointly integrable. 
This implies the existence of non-empty open sets of exponentially mixing Anosov flows.
As part of the proof of this result we show that $\cC^{1\plus}$ uniformly-expanding suspension semiflows (in any dimension) mix exponentially when the return time in not cohomologous to a piecewise constant. 
\end{abstract}

\maketitle
\thispagestyle{empty}

\section{Introduction \& Results}
Anosov flows~\cite{Anosov67}, which have been studied extensively since the 1960s, are arguably the canonical examples of chaotic dynamical systems and the rate of mixing (decay of correlation) is one of the most important statistical properties. Nevertheless our knowledge of the rate of mixing of Anosov flows remains unsatisfactory. The study of the rate of mixing for hyperbolic systems goes back to the work of Sinai~\cite{Sinai72} and Ruelle~\cite{Ruelle76} in the 1970s and plenty of results were obtained for maps during the subsequent years. However various results for flows have only been established relatively recently and several basic questions remain as open problems.
Exponential mixing is interesting in its own right, it is a intrinsic property of a dynamical system which describes the rate at which initial information is lost, but also it is crucial for establishing other quantitative statistical properties and work on more intricate models (prominently in nonequilibrium statistical mechanics, e.g., questions of energy transport~\cite{DL11}).

Let $\fl{t}: \cM \to \cM$ be an Anosov flow on $\cM$, a smooth compact connected Riemannian manifold. That \(\fl{t}\) is Anosov means that there exists a $\fl{t}$-invariant continuous splitting of tangent space $T\cM = \bE_s \oplus \bE_0 \oplus \bE_u$ where $\bE_0$ is the line bundle tangent to the flow, $ \bE_s$ is the stable bundle in which there is exponential contraction and $ \bE_u$ is the unstable bundle in which there is exponential expansion. It is known that each transitive Anosov flow admits a unique SRB measure which will be denoted $\mu$ (see \cite{Young02} for extensive information concerning SRB measures). This invariant measure is the one which is most relevant from a physical point of view. The focus of this text is to prove exponential mixing with respect to the SRB measure. 
By \emph{exponential mixing} we mean the existence of $C,\gamma>0$ such that
\(
\abs{\int_\cM f \cdot g\circ \fl{t} \ d\mu - \int_\cM f \ d\mu \int_\cM g\ d\mu }
\leq C \norm{f}_{\cC^1}  \norm{g}_{\cC^1} 
e^{-\gamma t} 
\)
for all $f,g \in \cC^1(\cM,\bR)$
and for all $t\geq 0$. 
(An approximation argument means that exponential mixing for $\cC^1$ observables implies also exponential mixing for H\"older observables \cite[Proof of Corollary 1]{Dolgopyat98}.)
In the following we will use the expression \emph{mixes exponentially} to mean with respect to the unique SRB measure for the flow, often without explicit mention of the measure.

Not all Anosov flows mix exponentially, indeed those which are constant time suspensions over Anosov maps are not mixing.\footnote{Suspensions over Anosov diffeomorphisms by a return time that is cohomologous to a constant are also not mixing but these can always be written as constant time suspensions.} One wonders if this degenerate case is the only way that Anosov flows can fail to mix exponentially or if other slower rates are possible. 
Taking a suspensions over an Anosov diffeomorphism is one way to construct Anosov flows but not all Anosov flows are of this type.
The geodesic flow of any compact Riemannian manifold of strictly negative curvature is an Anosov flow and these were a major motivation at the beginning of the study of Anosov flows. Some initial progress was made proving exponential mixing for geodesic flows in the case of constant curvature and low dimension (see the introduction of \cite{Liverani04} for details and further references) but these methods, which are group theoretical in nature, were not suitable for adaption to the general case of variable curvature, let alone to the question for Anosov flows which are not geodesic flows.

In the late 1990s a major advance was made by Dolgopyat~\cite{Dolgopyat98} who, building on the dynamical argument introduced by Chernov~\cite{Chernov98}, showed that transitive Anosov flows with $\cC^1$ stable  and unstable bundles mix exponentially whenever the stable and unstable bundles are not jointly integrable.\footnote{A \(k\)-dimensional subbundle is said to be \emph{integrable} if there exists a \(k\)-dimensional foliation whose leaves are tangent to the subbundle.} 
In particular this means that geodesic flows on surfaces of negative curvature mix exponentially (in this special case the regularity of the bundle is a result of the low dimension and the preserved contact structure which exists naturally for geodesic flows). 
However a question of foremost importance is to show that statistical properties hold for an open and dense set of systems and the problem here is that the requirement of regularity for both bundles simultaneously is not typically satisfied for  Anosov flows~\cite{HW99}. Both stable and unstable foliations are always H\"older but the regularity cannot in general be expected to be better than H\"older, a generic smooth perturbation\footnote{Here and in the following, by perturbation of the flow we mean a \(\cC^{r}\) (\(r\geq 1\)) perturbation of the vector field associated to the flow. The structural stability of Anosov flows means that such a perturbed vector field (under a small perturbation) also defines an Anosov flow.} will destroy the Lipschitz regularity of at least one of the foliations.\footnote{Stoyanov~\cite{Stoyanov11} obtained results similar to Dolgopyat~\cite{Dolgopyat98} for Axiom~A flows but, among other assumptions, required that local stable and unstable laminations are Lipschitz.} 

If a flow preserves a contact form then it is said to be a \emph{contact flow}.
Liverani~\cite{Liverani04} showed that all contact (with $\cC^2$ contact form) Anosov flows mix exponentially with no requirement on the regularity of the stable and unstable bundles. 
This provides a complete answer for geodesic flows on manifolds of negative curvature since all such geodesic flows are contact Anosov flows with smooth contact form.\footnote{Not every contact Anosov flow is a geodesic flow on a Riemannian manifold, for example the flows constructed by Foulon \& Hassleblatt~\cite{FH13}.}  
Liverani's requirement of a $\cC^2$ contact form has two important consequences: Firstly it guarantees that $\bE_{s}\oplus \bE_{u}$ is not integrable and this is a property which is robust under perturbation; Secondly the smoothness of the contact form guarantees the smoothness of the subbundle $\bE_{s}\oplus \bE_{u}$ and the smoothness of the temporal function \cite[Figure~2]{Liverani04}. This smoothness is essential to Liverani's argument. Unfortunately the existence of a $\cC^2$ contact form cannot be expected to be preserved by perturbations of the Anosov flow (the consequences of the existence of a smooth contact structure would contradict the prevalence of foliations with bad regularity which was mentioned above). 

In the case of Axiom~A flows\footnote{Axiom~A flows are a generalization of Anosov flows, they are uniformly hyperbolic but the maximal invariant set is permitted to be a proper subset of the underlying manifold (for further details see e.g., \cite{Bowen73}).} there exist flows which are mixing but mix arbitrarily slowly~\cite{Ruelle83}. These are constructed as suspensions over Axiom~A maps with piecewise constant (but not constant) return time and are consequently not Anosov flows. It would be interesting to understand if this phenomena can only exist in the Axiom~A case and not for Anosov flows. 

The Bowen-Ruelle conjecture states that every mixing Anosov flow mixes exponentially. At this present moment this conjecture remains wide open, there is a substantial distance between the above discussed results and the statement of the conjecture. One obvious possibility in order to proceed is to separate this conjecture into two separate conjectures: (A)~If an Anosov flow is mixing then \(\bE_s \oplus \bE_u\) is not integrable; (B)~A transitive Anosov flow mixes exponentially whenever \(\bE_s \oplus \bE_u\) is not integrable. A related, but seemingly slightly easier problem is to understand whether exponential mixing is an open and dense property for Anosov flows. 
Statement (A) was proved by Plante~\cite[Theorem 3.7]{Plante72} under the additional assumption that the Anosov flow is codimension one\footnote{An Anosov flow is said to be \emph{codimension-one} if  \(\dim \bE_s=1\) or \(\dim \bE_u=1\)} but the general statement remains an open conjecture. Our main aim is to show statement (B) in the greatest generality possible, i.e., to show exponential mixing under the assumption that  \(\bE_s \oplus \bE_u\) is not integrable.

The question of exponential mixing continues to be of significant importance, beyond the (rather special) setting of Anosov flows. In particular it would be easily argued that, from a physical point of view (e.g, the multitudes of uniformly hyperbolic billiard flows~\cite{CM06}), discontinuities are natural. In such situations part (B) in the above division of the conjecture is the important part.\footnote{In some settings (e.g., symbolic systems) it is not clear that the notion of integrability (or non-integrability) of \(\bE_s \oplus \bE_u\) always makes sense. However for Axiom~A attractors, using that unstable disks are contained within the maximal invariant set, the notion is fine and corresponds to the existence of a foliation of a neighbourhood of the attractor~\cite[\S3]{ABV14}. Another relevant direction is to consider dispersing billiard flows in the presence of a small external field.}
Given the Axiom~A examples mentioned above, it would be surmised that part (A) is a peculiarity of the special properties of Anosov flows.
The main advance to date for flows with discontinuities is the work of Baladi, Demers \& Liverani~\cite{BDL17} which proves exponential mixing for Sinai billiard flows (three-dimensional) and, as in the work mentioned above, their argument uses crucially the contact structure which is present in such billiard flows.

Major progress on exponential mixing for flows was made recently by Tsujii~\cite{Tsujii16} who demonstrated the existence of a \(\cC^{3}\)-open and  \(\cC^{r}\)-dense subset of volume-preserving three-dimensional Anosov flows which mix exponentially. Interestingly the set Tsujii constructs doesn't contain the flows which have \(\cC^1\) stable and unstable bundles (and consequently doesn't contain the flows which preserve a $\cC^2$ contact form).
In some sense the new ideas introduced in his work are the main recent advance towards settling the Bowen-Ruelle conjecture.
One of the consequences of this present text is to demonstrate that in certain higher dimensional settings the result analogous to Tsujii's can, to some extent, be proved rather more easily. 

It is enlightening to take a moment to consider the three-dimensional case in  more detail. As mentioned above it is known~\cite{Dolgopyat98,Liverani04} that any contact Anosov flow (and hence any geodesic flow of a negatively curved surface) mixes exponentially. Tsujii~\cite{Tsujii16} uses the expression ``twist of the stable subbundle along pieces of unstable manifolds'' to describe the geometric mechanism which produces exponential mixing for flows. For contact Anosov flows a key part of the argument, and a part which is clear in the work of Liverani~\cite{Liverani04}, is to use the contact structure to guarantee that (in the language of Tsujii) moving along the unstable manifold a prescribed distance guarantees a uniform amount of twist of the stable subbundle. On the other hand, Tsujii uses the fact that the twist ``will be `random' and `rough' in generic cases''. The core of our work described in this paper will be to study the flows by quotienting along stable manifolds. We will then take advantage of a twist in the sense discussed above but, since we have already quotiented, we will not distinguish between the two different cases.\footnote{In practice we will consider the picture with stable and unstable exchanged but this seems to be merely a preference and not significant when studying Anosov flows.} 

Given the evidence currently available it is reasonable to conjecture that (B) is true, i.e.,  transitive Anosov flows mix exponentially whenever \(\bE_s \oplus \bE_u\) is not integrable. However a complete solution of this problems appears to be a high order of difficulty and the path in this direction is not clear. It is also reasonable to hope that such holds more generally and that uniformly hyperbolic flows (with discontinuities permitted) mix exponentially whenever \(\bE_s \oplus \bE_u\) is not integrable (assuming sufficient structure such that integrability of this bundle has meaning). One of the motives behind this present work is to better understand and enlarge the set of Anosov flows which are known to be exponentially mixing in order to eventually improve our understanding of the general case.

At this point it is worth noting that the mechanism which is behind the exponential mixing of Anosov flows is also the mechanism which is important in some partially hyperbolic maps (see e.g., \cite[Appendix C]{DL17}) and is essential in semiclassical analysis  (see e.g., \cite{FT16}). 

Our first result concerns exponential mixing under relatively weak regularity assumptions.
\begin{thmmain}\label{thm:ExpStable}
 Suppose that $\fl{t} : \cM \to \cM$ is a transitive $\cC^{1\plus}$ Anosov flow\footnote{For any $k\in \bN$, the notation $\cC^{k\plus}$ means $\cC^{k+\alpha}$ for some $\alpha\in (0,1]$. That a flow is $\cC^{k\plus}$ is shorthand for requiring that the map \(\cM \times \bR \to \cM\); \((x,t) \mapsto \fl{t}x\) is \(\cC^{k\plus}\). } and that the stable bundle is $\cC^{1\plus}$. 
 If the stable and unstable bundles are not jointly integrable, then $\fl{t}$ mixes exponentially with respect to the unique SRB measure. 
\end{thmmain} 
\noindent
This result improves the result of Dolgopyat \cite{Dolgopyat98} since regularity is only required for the stable bundle whereas in the cited work regularity was required of both bundles. 
Although this change is small when measured in terms of the number of characters altered in the statement, we are required to completely redo the proof in a somewhat different fashion (even though the essential ideas behind the argument are the same).
More to the point, the improvement over Dolgopyat's previous result is substantial in terms of the advantage it gives in finding open sets of exponentially mixing flows. This is illustrated by the following theorem. 
\begin{thmmain}
\label{thm:ExpVolPres}
 Suppose that $\fl{t} : \cM \to \cM$ is a $\cC^{2\plus}$ volume-preserving Anosov flow and that \(\dim \bE_s=1\) and $\dim \bE_u\geq2$. 
 There exists a $\cC^1$-neighbourhood of this flow, such that, for all $\cC^{2\plus}$ Anosov flows in the neighbourhood, if the stable and unstable bundles are not jointly integrable, then the flow mixes exponentially with respect to the unique SRB measure. 
\end{thmmain}
\noindent
Since the set of Anosov flows where the stable and unstable bundles are  not jointly integrable is $\cC^{1}$-open and $\cC^{r}$-dense in the set of all Anosov flows (see \cite{FMT07} and references within concerning the prior work of Brin) the above theorem implies a wealth of open sets of exponentially mixing Anosov flows.
To the best of our knowledge, this is the first proof of the existence of open sets of Anosov flows which mix exponentially (observe that the neighbourhood in the statement of the theorem, although centred on a volume-preserving flow, is a neighbourhood in the set of all Anosov flows).
Similarly the set of Anosov flows where the stable and unstable bundles are  not jointly integrable is $\cC^{1}$-open and $\cC^{r}$-dense in the set of volume-preserving Anosov flows.\footnote{Consider a volume-preserving Anosov flow and assume that $\bE_s \oplus \bE_u$ is integrable. There exists a section such that the flow can be described as a suspension with constant return time. We will perturb the flow by smoothly modifying the magnitute of the associated vector field in a small ball. Following \cite{FMT07} we can do this in such a way to guarantee that, for the perturbed flow, $\bE_s \oplus \bE_u$ is not integrable. Note that the perturbed system is still Anosov and as smooth as before. Since we only changed the magnitude of the vector field the cross-section remains a cross-section and the return map also remains unchanged. Consequently we ensure that the perturbed flow also preserves a smooth volume.}
This means that an open and dense subset of the volume-preserving Anosov flows such that \(\dim \bE_s=1\) and $\dim \bE_u\geq2$ mix exponentially.
The ideas used here and the application of Theorem~\ref{thm:ExpStable} actually show exponential mixing for an even larger set of Anosov flows than stated in the above theorem but further details concerning this are postponed until the remarks in Section~\ref{sec:proofofthm2} (in particular we can prove the same conclusions in many cases where \(\dim \bE_s > 1\)). 

Let us consider the particular case of four-dimensional volume-preserving flows \(\fl{t} : \cM \to \cM\). Since the flow is Anosov and four-dimensional, either \(\dim \bE_s=1\) or \(\dim \bE_u=1\). In the first case Theorem~\ref{thm:ExpVolPres} applies directly. For the other case observe that the SRB measure for a volume-preserving Anosov flow is the preserved volume and consequently the SRB measure for the time reversed flow \(\fl{-t}\) is equal to the  SRB measure for \(\fl{t}\). 
Since \(\int_\cM f \cdot g\circ \fl{t} \ d\mu 
= \int_\cM f\circ \fl{-t} \cdot g \ d\mu \) and that stable and unstable are swapped for the time reversed flow we can again apply Theorem~\ref{thm:ExpVolPres}. Consequently the above result implies the following statement:
 Suppose that $\fl{t} : \cM \to \cM$ is a $\cC^{2\plus}$ four-dimensional volume-preserving Anosov flow.
 Then, if the stable and unstable bundles are not jointly integrable, the flow mixes exponentially with respect to the volume. In particular a $\cC^{1}$-open and $\cC^{r}$-dense subset of four-dimensional volume-preserving flows mix exponentially.
 This means that Tsujii's result holds in four-dimensions.
 As discussed above, Plante demonstrated that mixing implies that \(\bE_s \oplus \bE_u\) is not integrable in the codimension-one case. Consequently the results of this paper provide a complete resolution of the Bowen-Ruelle conjecture in the volume-preserving four-dimensional case.

\begin{rem}
The proof of Theorem \ref{thm:ExpVolPres} requires the flow to be transitive in order to apply Theorem~\ref{thm:ExpStable}. However, due to Verjovsky \cite{Verjovsky74}, 
codimension-one Anosov flows on higher dimensional manifolds (\(\dim \cM>3 \))
 are transitive and so transitivity is automatic\footnote{In the case where both the stable 
 and unstable bundles are at  least \(2\) dimensional  there are examples of non-transitive Anosov flows~\cite{FW80}. 
 Also, as it is remarked  in \cite{FW80}, the three dimensional case, where Verjovsky's proof does not work, this question of transitivity  of Anosov flows remains open.}
 in the case of Theorem~\ref{thm:ExpVolPres}.
\end{rem}

Section \ref{sec:Anosov} contains the proof of  Theorem \ref{thm:ExpStable} and the details of how Theorem \ref{thm:ExpVolPres} is derived from it.  
The proof of the first result rests heavily on a result (Theorem \ref{thm:ExpSemi} below) concerning exponential mixing for $\cC^{1\plus}$ expanding semiflows.  
Our motive for proving Theorem~\ref{thm:ExpSemi} was proving Theorem~\ref{thm:ExpStable} but Theorem~\ref{thm:ExpSemi} is also of interest in its own right. Details concerning past work on similar questions follows after we precisely introduce the setting.

We observe that the ideas in this text are very much limited to the argument presented here and will not suffice to fully answer the question of when in general Anosov flows mix exponentially. For this progress we hope that the work of Dolgopyat~\cite{Dolgopyat98}, Liverani~\cite{Liverani04}, Baladi \& Vall\'ee~\cite{BV05} and Tsujii~\cite{Tsujii16} (among others) can eventually be extended and improved.

We proceed by defining the class of  $\cC^{1\plus}$ expanding semiflows. Firstly we require two pieces of information concerning the geometry of the set. 
Let \(X\) be the disjoint union of a finite number of connected bounded open subsets of \(\mathbb{R}^d\) (we use the convention that the distance between two points in different connected components is infinite). 
\begin{defin}
 \label{def:john}
  We say that \(X \subset \bR^d\) is \emph{almost John} if there exist constants \(C,\epsilon_0>0, s\geq 1\) such that, for all \(\epsilon \in (0,\epsilon_0)\) and for all \(x\in X\), there exists \(y\in X\) such that \(\dist{x}{y} \leq  \epsilon\) and such that the ball centred at \(y\) of radius \(C\epsilon^s\) is contained in \(X\).\footnote{This condition on \(X\) is similar in spirit to the requirement of a John domain as used in~\cite{AGY06}. However they are not equivalent, in our case we need only weaker properties and so we can make do with weaker assumptions. See the discussion in Appendix~\ref{app:markov} for further details.}
\end{defin}

We will always assume that \(X\) is almost John and that the the boundary of \(X\) has upper box-counting dimension strictly less than \(d\).
Let $T:X\to X$ denote a \emph{uniformly expanding $\cC^{1\plus}$ Markov map}.
By this we mean that there exists $\cP$, a finite partition into connected open sets of a full measure subset of $X$ such that, for each $\omega \in \cP$, $T$ is a $\cC^{1}$ diffeomorphism from $\omega$ to $T\omega$  and that $T\omega$ is a full measure subset of one of the connected components of \(X\).\footnote{I.e., the map is required to be Markov but it is not necessarily full-branch.}
\begin{rem}
The conditions on \(X\) would be satisfied if the boundary of \(X\) were a finite union of $\cC^{1}$-submanifolds. However, in view of the intending application application, we must allow lower regularity of the boundary since such low regularity is the unfortunate reality for Markov partitions~\cite{Bowen78}. 
\end{rem}
We require that there exist $\cexpand >0$, $\lambda >0$ such that
\begin{equation}
\label{eq:Expanding}
 \norm{(DT^{n}(x))^{-1}} \leq  \cexpand e^{-\lambda n}
 \quad \text{for all $x\in X$, $n\in \bN$},
\end{equation}
and there exist $\cdistor>0$, $\alpha \in (0,1)$ such that
\begin{equation} 
 \label{eq:DisControl}
 \abs{\ln \frac{\det(DT(x))}{\det(DT(y))}} \leq \cdistor \dist{Tx}{Ty}^\alpha
 \quad \text{ for all $\omega\in \cP$, for all $x,y \in \omega$}.
\end{equation}
We also require $T$ to be covering in the sense that for every open ball $B \subset X$ there exists $n\in \bN$ such that $T^{n}B = X$ (modulo a zero measure set).
For such maps it is known that there exists a unique $T$-invariant probability measure absolutely continuous with respect to Lebesgue. We denote this measure by $\nu$. The density of the measure is H\"older (on each partition element) and bounded away from zero.
Let $\tau: X \to \bR_+$ denote the \emph{return time function}. 
We require that $\tau$ is $\cC^{1+\alpha}$, that there exists $\ceta >0$ such that\footnote{In our setting \eqref{eq:TauControl} could be simplified by removing \(DT(x)^{-1}\) from the equation. We choose to write it like this because this is the quantity which occurs naturally. }
\begin{equation}
\label{eq:TauControl}
 \norm{D\tau(x) DT(x)^{-1} } \leq \ceta 
 \quad \text{for all $x\in \omega$,  $\omega\in \cP$},
\end{equation}
and that there exists $\cBoundTau>0$ such that
\begin{equation} 
 \label{eq:BoundTau} 
 \tau(x) \leq \cBoundTau
  \quad \text{for all $x\in \omega$,  $\omega\in \cP$}.
\end{equation}
The suspension semiflow $T_{t} : X_\tau \to X_\tau$ is defined as usual, $X_\tau :=\{(x,u): x\in X, 0\leq u<\tau(x)\}$ and $T_t : (x,u) \mapsto (x,a+t)$ modulo the identifications $(x,\tau(x)) \sim (Tx, 0)$. The unique absolutely continuous $T_{t}$-invariant probability measure\footnote{$\nu_\tau(f) = \frac{1}{\nu(\tau)} \int_X\int_0^{\tau(x)} f(x,u) \ du \ d\nu(x)$} is denoted by $\nu_\tau$.

Baladi and Vall\'ee~\cite{BV05} showed that semiflows similar to above, but with the $\cC^2$ version of assumptions, typically mix exponentially when $X$ is one dimensional. The same argument was shown to hold by Avila, Gou\"ezel \& Yoccoz~\cite{AGY06}, again in the $\cC^2$ case, irrespective of the dimension of $X$. Recently Ara\'ujo \& Melbourne~\cite{AM15} showed that the argument still holds in the $\cC^{1\plus}$ case when $X$ is one dimensional. This weight of evidence means that the following result is not unexpected.
\begin{thmmain}
 \label{thm:ExpSemi}
 Suppose that $T_t : X_\tau \to X_\tau$ is a uniformly expanding $\cC^{1\plus}$ suspension semiflow as above. 
 Then either $\tau$ is cohomologous to a piecewise constant function or there exists $C, \gamma>0$ such that, for all $f,g \in \cC^1(X_\tau,\bR)$, \(t\geq 0\),
\[
\abs{\int_{X_\tau} f \cdot g\circ T_{t} \ d\nu_\tau - \int_{X_\tau} f \ d\nu_\tau \int_{X_\tau} g\ d\nu_\tau }
\leq C \norm{f}_{\cC^1}  \norm{g}_{\cC^1}  e^{-\gamma t}. 
\] 
\end{thmmain}
\noindent
The proof of the above is the content of Section \ref{sec:SemiFlow}.
The estimate for exponential mixing relies on estimates of the norm of the twisted transfer operator given in Proposition~\ref{prop:FuncResult}.
In some sense Proposition~\ref{prop:FuncResult} is the main result of this part of the paper and the exponential mixing which we use here is merely one consequence of it. For many other applications, for example, other statistical properties or the study of perturbations, the extra information contained in the functional analytic result is key. However we avoid giving the statement here because it relies on a significant amount of notation which is yet to be introduced.  

The argument of \cite{AM16} follows closely the argument of \cite{AGY06} which in turn follows closely the argument of \cite{BV05}. Everything suggests that exactly this argument could be used with minor modification in order to prove Theorem~\ref{thm:ExpSemi}. That the structure of the proof contained in Section~\ref{sec:SemiFlow} is superficially rather different is merely due to the aesthetic opinion of the present authors.

\section{Anosov Flows}\label{sec:Anosov}
This section is devoted to the proof of Theorem~\ref{thm:ExpStable} and Theorem~\ref{thm:ExpVolPres}.  The proof of Theorem~\ref{thm:ExpStable} relies crucially on Theorem \ref{thm:ExpSemi}.  The proof of Theorem~\ref{thm:ExpVolPres} relies crucially on Theorem \ref{thm:ExpStable}.
\subsection{Proof of Theorem \ref{thm:ExpStable}}
 Suppose that $\fl{t} : \cM \to \cM$ is a $\cC^{1+\alpha}$ Anosov flow and that the stable bundle is $\cC^{1+\alpha}$ for some $\alpha>0$. 
The proof is based (as per \cite{AV12, ABV14, AM15}) on quotienting along local stable manifolds and reducing 
the problem to the study of the corresponding expanding suspension semiflow. We then use the estimate which is given by Theorem~\ref{thm:ExpSemi}.

The argument is the same idea as used previously~\cite{ABV14} for Axiom A flows.\footnote{Lemma 5 in \cite{ABV14} contains an inaccuracy: there it is claimed that the domain of the uniformly expanding map is a \(\cC^2\) disk whereas the reality is that it is a subset of such disks but with a boundary of poor smoothness.} The only difference being that some of the estimates are now H\"older and not $\cC^1$ since here we have merely a $\cC^{1+\alpha}$ stable bundle whereas in the reference the bundle is $\cC^2$. 
One important consideration in this argument is the regularity of the boundary of the elements of the Markov partition. Appendix~\ref{app:markov} is devoted to further details concerning the construction and various important estimates which will be required, in particular estimate concerning the boundary of elements of the partition.

We recall that Bowen constructed~\cite{Bowen70} Markov partitions for Axiom~A diffeomorphisms and then extended~\cite{Bowen73} this construction to Axiom~A flows, in particular for Anosov flows. Ratner~\cite{Ratner73} also constructed Markov partitions for Anosov flows, again based on Bowen's previous work. We will take Ratner's description of the construction as our primary reference since several parts of that presentation are more amenable to our present purposes.  

The main idea is that we can find a section which consists of a family of local sections which are $\cC^{1+\alpha}$ and foliated by local stable manifolds. 
The return map is a uniformly hyperbolic Markov map on the family of local sections~\cite{Bowen73}. 
Let $Y$ denote the union of the local sections and let $S: Y \to Y$ and $\tau: Y \to \bR_+$ denote the return map and return time for $\fl{t}$ to this section. 
Let $\eta$ denote the unique SRB measure for $S: Y \to Y$. Note that $\tau$ is constant~\cite[\S3]{ABV14} along the local stable manifolds. 

We now quotient along the local stable manifolds (within the local sections) letting $\pi: Y \to X$ denote the quotient map. 
Consequently we obtain a map $T : X \to X$ such that $T\circ \pi = \pi \circ S$. 
Since the original flow is Anosov (in particular an attractor) the set \(X\) is the finite union of connected components. Each connected component is a subset of a \(\cC^{1+\alpha}\)  submanifold of the same dimension as the unstable bundle. However the boundary of these components, viewed as a subset of this submanifold, cannot be expected to be smooth~\cite{Bowen78}. 

That the assumptions on \(X\)  which are required by Theorem~\ref{thm:ExpSemi} are satisfied is shown in Section~\ref{app:DimOfBoundary} and Lemma~\ref{lem:fillball}. 
Because of the properties of $S$ (in particular due to the use of the Markov partition in the above construction), the map $T$ is a uniformly expanding Markov map and satisfies the conditions \eqref{eq:Expanding}, \eqref{eq:DisControl}, \eqref{eq:TauControl} and \eqref{eq:BoundTau}. Therefore, applying 
Theorem \ref{thm:ExpSemi}, we have that either the suspension semiflow \(T_{t}\) mixes exponentially or \(\tau\) is cohomologous to a constant function. If \(\tau\) is cohomologous to a constant function then~\cite[Lemma 12]{ABV14} the stable and unstable bundles are jointly integrable so for the rest of the proof, we suppose that \(\tau\)  is not cohomologous to a constant function and hence \(T_{t}\) mixes exponentially.
 
Let $\nu$ denote the unique SRB measure for $T$ ($\nu = \pi_* \eta$). To proceed we observe that the measure $\nu$ admits a disintegration into conditional measures along local stable manifolds. We observe~\cite{BM15} that there exists a family of conditional measures ${\{  \nu_{x} \}}_{x\in X}$ ($\nu_{x}$ supported on $\pi^{-1}x$) such that
\[
 \eta(v) = \int_{X} \nu_{x}(v) \ d\eta(x)  
\]
for all continuous functions $v : Y \to \bR$.
We also know~\cite[Proposition 6]{BM15} that this disintegration has good regularity in the sense that $x \mapsto \nu_{x}(v)$ is H\"older on each partition element and has uniformly bounded H\"older norm for any H\"older $v:Y\to\bR$.

Let $Y_\tau$, $S_t$, $\eta_\tau$ be defined analogously to $X_\tau$, $T_t $, $\nu_\tau$.
Suppose $u,v : Y_\tau \to \bR$ are H\"older continuous functions. Points in \( Y\) are  denoted by \((x,a)\) which is given by the product representation of \(Y\) 
by \(X\) times the local stable manifolds. To prove that \(S\) mixes exponentially, it is convenient to write
\begin{equation}
\label{eq:Split}
\int_{Y_\tau} u \cdot v \circ S_{2t} \ d\eta_\tau
=
\int_{Y_\tau} u \cdot (v \circ S_{t} - v_t\circ \pi_\tau) \circ S_t \ d\eta_\tau
+
\int_{X_\tau} \tilde u \cdot v_t \circ T_{t} \ d\nu_\tau
\end{equation}
where $\tilde u : X_\tau \to \bR$,  $v_t : X_\tau \to \bR$ are defined as
\[
\tilde u(x,a) := \int_{\pi^{-1}x} u(y,a) \ d\nu_x(y), \quad
v_t(x,a)  := \int_{\pi^{-1}x} v\circ S_t(y,a) \ d\nu_x(y).
\]
The new observables \(\tilde u\) and \(v_t\) are \(\mathcal{C}^{\alpha}\) on each partition element as observed above.
To estimate the first term of \eqref{eq:Split} we observe that 
\[
(v \circ S_{t} - v_t\circ \pi_\tau)(y,u) 
=
\int_{\pi^{-1}(\pi y)} v \circ S_{t}(y,u) -  v\circ S_t(z,u) \ d\nu_{\pi y}(z).
\]
Consequently the function \(v_t\) is exponentially close to \(v\circ S_t\) on each local stable manifold
and so 
\begin{equation}\label{eq:FirstTerm}
\left|\int_{Y_\tau} u \cdot (v \circ S_{t} - v_t\circ \pi_\tau) \circ S_t \ d\eta_\tau\right|\leq C\norm{u}_{\cC^{\alpha}}\norm{v}_{\cC^{\alpha}} e^{-\tilde\gamma t}
\end{equation}
where \(\tilde\gamma >0\) depends on the contraction rate on the stable bundle.

The second term of \eqref{eq:Split} is estimated using  Theorem~\ref{thm:ExpSemi} which says that \(T_{t}\) mixes exponentially since $\tau$ is not cohomologous to a piecewise constant. We have
\begin{equation}\label{eq:SecondTerm}
\left|\int_{X_\tau} \tilde u \cdot v_t \circ T_{t} \ d\nu_\tau-\int_{X_\tau}\tilde u \ d\nu_\tau\cdot\int_{X_\tau}v_t \ d\nu_\tau\right|\leq C\norm{\tilde u}_{\cC^{\alpha}}\norm{v_t}_{\cC^{\alpha}} e^{-\gamma t}.
\end{equation}
Using estimates \eqref{eq:FirstTerm} and \eqref{eq:SecondTerm} in \eqref{eq:Split} gives that the flow \(S_{t}: Y_{\tau}\to Y_{\tau}\) mixes exponentially. This in turn implies that the flow \(\fl{t}\) is exponentially mixing.

\subsection{Proof of Theorem \ref{thm:ExpVolPres}}
\label{sec:proofofthm2}
The proof consists of showing that if \(\fl{t}\) is $\cC^{1}$-close to a volume preserving flow and that \(\dim \bE_s=1\), $\dim \bE_u\geq2$ then the stable 
bundle is \(\mathcal{C}^{1\plus}\). We then apply Theorem \ref{thm:ExpStable}.

We recall that the regularity of the invariant bundle of an Anosov flow is given by Hirsch, Pugh \& Shub~\cite{HPS77} (see also \cite[Theorem~4.12]{AM16}) under the following \emph{bunching} condition. Suppose that $\fl{t}:\cM\to\cM$ is a $\cC^{2\plus}$ Anosov flow\footnote{This is the only place where the flow is required to be $\cC^{2\plus}$, everywhere else $\cC^{1\plus}$ suffices.}. If there exists $t,\alpha>0$ such that
\begin{equation}
\label{eq:Bunching} 
\sup_{x\in\cM} \norm{ \smash{ \left. D\fl{t} \right|_{\bE_s}(x) } }\norm{ \smash{ \left. D\fl{t} \right|_{\bE_{cu}}^{-1}}(x)  }\norm{ \smash{ \left. D\fl{t} \right|_{\bE_{cu}}(x) } }^{1+\alpha}  <1,
\end{equation}
then the stable bundle is $\cC^{1+\alpha}$ ($\bE_{cu} = \bE_{u} \oplus \bE_{0}$ and is called the \emph{central unstable sub-bundle}).

Following Plante~\cite[Remark 1]{Plante72}, we observe that, in the case when the Anosov flow is volume preserving, $\dim \bE_s = 1$ and $\dim \bE_u \geq 2$, then the 
above bunching condition holds true and consequently that the stable bundle is $\cC^{1+\alpha}$ for some $\alpha >0$.
This is because volume-preserving means that the contraction in $\bE_{s}$ must equal the volume expansion in $\bE_{u}$. Since  $\dim \bE_u \geq 2$ the maximum expansion in any given direction must be dominated by the contraction.
Consequently the stable bundle is $\cC^{1+\alpha}$. 
From its definition the bunching condition \eqref{eq:Bunching} is robust under $\cC^1$ perturbations of the Anosov flow.

\begin{rem}
This argument for the robust regularity of the stable bundle
 uses crucially that the unstable bundle has dimension at least 2 whilst the stable bundle has dimension 1. Such an argument is therefore not possible if the Anosov flow is three dimensional (see \cite{Plante72} for a counter example).
 Of course regular bundles are possible in the three-dimensional case but not in a robust way.
 \end{rem}
 
 \begin{rem}
In general, when $\dim \bE_s < \dim \bE_u $  it is again possible to find open sets such that the bunching condition is satisfied although this will not be possible for all such flows. A natural assumption to add would be isotropy of the hyperbolicity, i.e., that the expansion is of equal strength in all directions and similarly for the contraction. In this case we can again obtain \eqref{eq:Bunching} robustly and prove the analog of Theorem~\ref{thm:ExpVolPres}.
\end{rem}

\begin{rem}
In higher dimensions, with a large difference between the dimensions of the stable and unstable bundles, it is sometimes possible to obtain stronger bunching and therefore to guarantee that the stable bundle is \(\cC^2\) in a robust way. In this case results for \(\cC^2\) expanding semiflows~\cite{AGY06} can be  applied according to the same argument as in this paper and exponential mixing proved for the flow~\cite{ABV14}. A substantial part of this paper is to prove Theorem~\ref{thm:ExpSemi} which generalises prior work to the higher dimensional \(\cC^{1\plus}\) case. This is required to be able to handle a significantly larger set of Anosov flows, in particular to hold for any flow in dimension 4 and higher when \(\dim \bE_s = 1\).
\end{rem}

\section{{Expanding Semiflows}}
\label{sec:SemiFlow}
This section is devoted to the proof of Theorem~\ref{thm:ExpSemi}.
Throughout the section we suppose the setting of the theorem.
Recall that the semiflow is a combination of a uniformly expanding map $T : X \to X$ and return time $\tau: X \to \bR_{+}$. 
Let $m$ denote Lebesgue measure on $X$.
We will assume, scaling if required, that the diameter of $X$ is not greater than $1$ and that $m(X) \leq 1$.
We will also assume that $\cexpand=1$ in assumption \eqref{eq:Expanding}.
Suppose that this is not the case originally, then there exists some iterate such that $\cexpand e^{-\lambda n} <1$. We choose some partition element such that returning to this partition element takes at least $n$ iterates.  We take $\tilde X$ (which will replace $X$) to be equal to this partition element and choose for $\tilde T$ the first return map to $\tilde X$. The new return time $\tau$ is given by the corresponding sum of the return time. There is then a one-to-one correspondence between the new suspension semiflow and the original. It is simply a different choice of coordinates for the flow which has the effect that the expansion per iterate is increased and the return time increases correspondingly. This is not essential but it is convenient because below we can choose a constant conefield which is invariant.  
We will also assume for notational simplicity that $\cBoundTau \leq 1$, i.e., that $\tau(x) \leq 1$ for all $x$. This can be done without loss of generality, simply by scaling uniformly in the flow direction.
Let $\Lambda>0$ be such that $\norm{ DT(x)} \leq e^{\Lambda} $ for all $x$. This relates to the maximum possible expansion whereas $\lambda >0$ relates to the minimum expansion. 
After these considerations the suspension semiflow is controlled by the constants $\alpha \in (0,1)$, $\Lambda \geq \lambda>0$ and $ \cdistor, \ceta >0$.

Central to the argument of this section are Proposition~\ref{prop:Transversal}, Proposition~\ref{prop:Dolgopyat} and Proposition~\ref{prop:FuncResult}. The first describes how we see, in an exponential way, a key geometric property.
The second proposition uses this geometric property and the idea of oscillatory integrals in order to see cancellations on average. The third proposition is the combination of the previous estimates to produce the key estimate on the norm of the twisted operators.

\subsection{Basic Estimates}
Let $\ctaup = 2 \ceta/(1-e^{-\lambda})$,
let $\tau_{n} := \sum_{j=0}^{n-1} \tau \circ T^{j}$
and let $\cP_n$ denote the $n$\textsuperscript{th} refinement of the partition.
For convenience we will systematically use the notation $\ell_{\omega} := (\left. T^{n}\right|_{\omega})^{-1}$ for any $n\in \bN$, $\omega \in \cP_{n}$.
Let  $J_n(x) = 1/ \det DT^n(x)$.

\begin{lem}
\label{lem:TauPrime}
$\norm{ D(\tau_{n}\circ \ell_{\omega})(x) } \leq \frac{1}{2} \ctaup$ for all $n\in \bN$,  $\omega \in \cP_n$, $x\in T^{n}\omega$.
\end{lem}
\begin{proof}
Let $y = \ell_{\omega}(x)$ and observe that
\[
D(\tau_{n}\circ \ell_{\omega})(x) = \sum_{k=0}^{n-1} D\tau(T^{k}y) D(T^{k} \circ\ell_{\omega})(T^{k}y)
\] 
 Consequently, using also  \eqref{eq:Expanding}  and \eqref{eq:TauControl}, $\norm{ D(\tau_{n}\circ \ell_{\omega}) }  \leq  \ceta \sum_{k=0}^{n-1} e^{-\lambda(n-k)}$.
As $\sum_{k=0}^{\infty} e^{-\lambda k} = (1- e^{-\lambda})^{-1}$ the required estimate holds.
\end{proof}

\begin{lem}
 \label{lem:DistortionA}
 There exists $\cdistorn>0$ such that, for all $n\in \bN$, $\omega\in \cP_n$
\[
 \abs{  \ln \frac{ \det(D\ell_{\omega}(x)) }{ \det (D\ell_{\omega}(y)) } }
 \leq \cdistorn \dist{x}{y}^\alpha
 \quad \text{for all $x,y \in T^{n}\omega$}.
\]
\end{lem}
\begin{proof}
 We write $\ell_{\omega} = g_1 \circ \cdots \circ g_{n}$ where each $g_k$ is the inverse of $T$ restricted to the relevant domain.
 Let $x_{k} = T^{k}\ell_{\omega} x$, $y_{k} = T^{k}\ell_{\omega} y$.
 Consequently $\det(D\ell_{\omega}(x)) = \prod_{k=1}^n \det(Dg_k(x_{k}))$ and so
 \[
   \abs{  \ln \frac{ \det(D\ell_{\omega}(x)) }{ \det (D\ell_{\omega}(y)) } }
   \leq
   \sum_{k=1}^n
    \abs{  \ln \frac{ \det(Dg_k(x_{k})) }{ \det(Dg_k(y_{k})) } }.
 \]
 Assumption \eqref{eq:DisControl} implies that 
$ \abs{  \ln \frac{ \det(Dg_k( x_{k})) }{ \det(Dg_k(y_{k})) } }
    \leq
 \cdistor \dist{x_{k}}{y_{k}}^{\alpha}
 $.
 Using also assumption \eqref{eq:Expanding}  we obtain a bound  $\sum_{k=1}^n \cdistor ( e^{-\lambda ({n-k})})^\alpha \dist{x}{y}^{\alpha}$.  To finish the estimate let $\cdistorn:= \cdistor \sum_{j=0}^\infty e^{- \lambda {\alpha j}}$.
\end{proof}

\begin{lem}
 \label{lem:DistortionB}
 There exists $\cbalance>0$ such that
  \[
  \sum_{\omega\in \cP_n} \norm{J_{n}}_{L^{\infty}(\omega)} \leq \cbalance
  \quad \text{for all $n\in \bN$}.
 \]
\end{lem}
\begin{proof}
 For each $\omega\in \cP_n$ there exists some $x_\omega \in \omega$ such that $m(\omega) =  J_{n}(x_\omega) m(T^{n}\omega)$. This  means that $\sum_{\omega\in \cP_n}  J_{n}(x_\omega)) \leq m(x) \left(\inf_{\omega} m(T^{n}\omega)\right)^{-1}$.
 By Lemma~\ref{lem:DistortionA}  
 \[
\norm{J_{n}}_{L^{\infty}(\omega)} / J_{n}(x_\omega)
\leq e^{\cdistorn}.
\]
Consequently $ \sum_{\omega\in \cP_n} \norm{J_{n}}_{L^{\infty}(\omega)} \leq \cbalance$ where $\cbalance:= e^{\cdistorn} / \inf_{\omega} m(T^{n}\omega)$. 
\end{proof}

\subsection{Twisted Transfer Operators}
For $z\in \bC$, the twisted transfer operator $\cL_z : L^{\infty}(X) \to L^{\infty}(X)  $ is defined as
\[
 \cL_z^{n} f = \sum_{\omega \in \cP_{n}} (e^{-z\tau_{n}} \cdot f \cdot J_{n})\circ \ell_{\omega} \cdot \one_{T^{n}\omega}.
\]
We use the standard notation for the H\"older seminorm $\abs{f}_{\cC^{\alpha}(J)}$ where $J$ is any metric space. I.e.,  $\abs{f}_{\cC^{\alpha}(J)}$ is the supremum of $C\geq 0$ such that $\abs{f(x)-f(y)} \leq C \dist{x}{y}^\alpha$ for all $x,y \in J$, $x\neq y$. The H\"older norm is defined $\norm{f}_{\cC^{\alpha}(J)} := \abs{f}_{\cC^{\alpha}(J)} + \norm{f}_{L^{\infty}(J)}$.
Recall that \(X\) is the disjoint union of a finite number of connected subsets of \(\bR^d\). In this case
\[
\abs{f}_{\cC^{\alpha}(X)} :=  \sup_{\substack{x,y}} \frac{\abs{f(x)-f(y)}}{\dist{x}{y}^\alpha}
\]
where the supremun is taken over all \(x,y \in X\) which are in the same connected component as each other and \(x\neq y\).
As before let $\norm{f}_{\cC^{\alpha}(X)} := \abs{f}_{\cC^{\alpha}(X)} + \norm{f}_{L^{\infty}(X)}$.
Let ${\cC^{\alpha}(X)} := \{ f:X \to \bR : \abs{f}_{\cC^{\alpha}(X)} <\infty  \}$. 
This is a Banach space when equipped with the norm $\norm{\cdot}_{\cC^{\alpha}(X)}$.
Define, for all $b\in \bR$, the equivalent norm 
\[
\norm{f}_{(b)}:= \tfrac{1}{(1+\abs{b}^\alpha)} \abs{f}_{\cC^{\alpha}(X)} + \norm{f}_{L^{\infty}(X)}.
\]
Observe that, using Lemma~\ref{lem:DistortionB}, $\norm{\cL^{n}_{z}f}_{L^{\infty}(X)} \leq \cbalance e^{-\Re(z)n} \norm{f}_{L^{\infty}(X)}$ for all $n\in \bN$, $f\in L^{\infty}(X)$.

The argument of this section depends on choosing $\sigma>0$ sufficiently small in a way which depends only on the system $(X,T,\tau)$. We suppose from now on that such a $\sigma>0$ is fixed (sufficiently small) and the precise constraints on \(\sigma\) will appear at the relevant places in the following paragraphs.

\begin{lem}
 \label{lem:LY}
There exists $\cly >0$ such that, for all $z=a+ib$, $a>-\sigma$, $f \in \cC^\alpha(X)$, $ n\in \bN$,
\[
 \norm{\cL^n_z f}_{\cC^{\alpha}(X)} \leq \cly e^{-( \alpha\lambda - \sigma) n} \abs{f}_{\cC^{\alpha}(X)}
 + \cly e^{\sigma n} (1+\abs{b}^\alpha)  \norm{f}_{L^{\infty}(X)}.
\]
\end{lem}
\begin{proof}
Suppose that $\omega\in \cP_n$, $f \in \cC^\alpha(X)$ and $x,y \in T^{n}\omega$, $x\neq y$, then
\[
   (e^{-z\tau_n} \cdot f  \cdot  J_n)(\ell_{\omega}x) 
  - (e^{-z\tau_n}  \cdot f  \cdot  J_n)(\ell_{\omega}y)
  = A_1 + A_2 + A_3 + A_4
\]
where
\[
 \begin{aligned}
A_1 &=
(e^{-ib\tau_n(\ell_{\omega}x)} - e^{-ib\tau_n(\ell_{\omega}y)}) (e^{-a\tau_n} \cdot  f   \cdot J_n)(\ell_{\omega}x) \\
A_2 &=
e^{-ib\tau_n(\ell_{\omega}y)}  ( e^{-a\tau_n(\ell_{\omega}x)} - e^{-a\tau_n(\ell_{\omega}y)}) (f  \cdot  J_n)(\ell_{\omega}x)\\
A_3 &=
   e^{-z\tau_n(\ell_{\omega}y)}(f(\ell_{\omega}x)-f(\ell_{\omega}y)) \cdot J_n(\ell_{\omega}x) \\
A_4 &=
   e^{-z\tau_n(\ell_{\omega}y)}f(\ell_{\omega}y) (J_n(\ell_{\omega}x)-J_n(\ell_{\omega}y)).
 \end{aligned}
\]
By Lemma~\ref{lem:TauPrime}
$
  \abs{A_1} 
\leq ( e^{-a\tau_n} \cdot   \abs{f} \cdot J_n)(\ell_{\omega}x) 2 \min(\abs{b} \frac{\ctaup}{2} \dist{x}{y},1)$.
Since $\min(u,1) \leq u^\alpha$ for all $u\geq 0$,
 $\abs{A_1} 
\leq ( e^{-a\tau_n} \cdot   \abs{f} \cdot J_n)(\ell_{\omega} x) 2 \abs{b}^\alpha (\frac{\ctaup}{2})^\alpha \dist{x}{y}^\alpha$.
Again, by Lemma~\ref{lem:TauPrime},
\[
\begin{aligned}
  \abs{A_2} 
  &\leq e^{-a\tau_n(\ell_{\omega}x)}
  \abs{\smash{1 - e^{-a(\tau_n(\ell_{\omega}y) -   \tau_n(\ell_{\omega}x))}}}
    (\abs{f} \cdot J_n)(\ell_{\omega}x) \\
  & \leq ( e^{-a\tau_n} \cdot   \abs{f} \cdot J_n)(\ell_{\omega}x) \abs{a} \tfrac{\ctaup}{2}  \dist{x}{y}.
\end{aligned}
 \]
Using assumption \eqref{eq:Expanding}
$
 \abs{A_3}
 \leq ( e^{-a\tau_n} \cdot J_n)(\ell_{\omega}y) e^{-\alpha \lambda n} \dist{x}{y}^\alpha \abs{f}_{\cC^{\alpha}(\omega)}
 $.
Finally, by Lemma~\ref{lem:DistortionA} 
$
   \abs{A_4} 
   \leq
   ( e^{-a\tau_n} \cdot   \abs{f} \cdot J_n)(\ell_{\omega} y) \cdistorn \dist{x}{y}^\alpha
   $.
Summing over $\omega \in \cP_{n}$ we obtain
\begin{multline}
\frac{ \abs{\cL_z^n f(x) -  \cL_z^n f(y)}}{ \dist{x}{y}^\alpha}\\
 \leq
 \norm{ \cL_a^n 1 }_{L^{\infty}(X)} \left[ (  (2\abs{b}^{\alpha} + \abs{a}) \frac{\ctaup}{2}  + \cdistorn)  \norm{f}_{L^{\infty}(X)}    +  \cbalance  e^{-\lambda n}   \abs{f}_{\cC^{\alpha}(X)}  \right]
\end{multline}
To finish the estimate we observe that $\norm{\cL^{n}_{z}f}_{L^{\infty}(X)} \leq \norm{\cL^{n}_{\sigma}1}_{L^{\infty}(X)} \norm{f}_{L^{\infty}(X)}$, that \( \norm{\cL^{n}_{\sigma}1}_{L^{\infty}(X)} \leq   \cbalance e^{\sigma n}\) and choose $\cly$ according to the above equation.
\end{proof}

Lemma~\ref{lem:LY}, observing the definition of the $\norm{\cdot}_{(b)}$ norm, implies the following uniform estimate.
\begin{lem}
 \label{lem:AdaptedLY}
For all $z=a+ib$, $a >-\sigma$,
\[
 \norm{\cL^n_z f}_{(b)} \leq \cly e^{\sigma n} \left( e^{-\lambda n} \norm{f}_{(b)}
 +   \norm{f}_{L^{\infty}(X)}\right)
 \quad \text{for all $f \in \cC^\alpha(X)$, $ n\in \bN$}.
\]
\end{lem}

\subsection{Exponential transversality} 
\label{sec:Transversal}
The goal of this subsection is to prove Proposition~\ref{prop:Transversal} below. This is an extension of Tsujii~{\cite[Theorem~1.4]{Tsujii08}} to the present higher dimensional situation.
Much of the argument follows the reasoning of the above mentioned reference with some changes due to the more general setting.

Define the $(d+1)$-dimensional square matrix $\cD^n(x): \bR^{d+1} \to \bR^{d+1}$, 
\[
\cD^n(x) = \begin{pmatrix} DT^n(x) & 0 \\ D\tau_n(x) & 1 \end{pmatrix}.
\]
This is notationally convenient since $DT_t(x,s) = \cD^n(x)$ whenever $\tau_n(x) \leq s+t < \tau_{n+1}(x)$.\footnote{If one wished to study the skew-product $G: (x,u) \mapsto (Tx, u-\tau(x))$ this is also the relevant object to study since $\cD^n = DG^n$.}
To proceed it is convenient to establish the notion of an invariant unstable cone field. 
Recall that $\ctaup = 2 \ceta/(1-e^{-\lambda})$. We define   $\cK \subset \bR^{d+1}$ as 
\[
\cK = \left\{ \left(\begin{smallmatrix} \VecA \\ \VecB \end{smallmatrix}\right) : \VecA \in \bR^d, \VecB \in \bR, \abs{\VecB} \leq \ctaup \abs{\VecA}  \right\}.
\]
We refer to $\cK$ as a \emph{cone}.
We will see now that the width of the cone has been chosen sufficiently wide to guarantee invariance.  
Note that
\[
 \begin{pmatrix}
  DT(x) & 0 \\ D\tau(x) & 1
 \end{pmatrix}
 \begin{pmatrix}
  \VecA \\ \VecB
 \end{pmatrix}
 =
 \begin{pmatrix}
  DT(x) \VecA \\ D\tau(x) \VecA + \VecB
 \end{pmatrix}
 =
 \begin{pmatrix}
  \VecA' \\ \VecB'
 \end{pmatrix} 
\]
Let $\omega\in \cP$ be such that $\VecA = D\ell_{\omega}(Tx) \VecA'$.
Using conditions \eqref{eq:Expanding} and \eqref{eq:TauControl}, we have
\begin{multline}
 \label{eq:Invariant}
  \abs{\VecB'}
 = \abs{D\tau(x) \, \VecA + \VecB}
 = \abs{D(\tau\circ \ell_{\omega})(Tx) \, \VecA' + \VecB}\\
 \leq \ceta \abs{\VecA'} + \ctaup e^{-\lambda} \abs{\VecA'}
 \leq \tfrac{1}{2} \ctaup \abs{\VecA'}.
\end{multline}
Suppose that $x_1,x_2 \in X$, $n\in \bN$ such that $T^nx_1 = T^nx_2$. 
We write  $$\cD^n(x_1)\cK \pitchfork \cD^n(x_2)\cK$$ if $\cD^n(x_1)\cK \cap \cD^n(x_2)\cK$ does not contain a $d$-dimensional linear subspace. In such a case we say that the image cones are \emph{transversal}.

\begin{prop}
\label{prop:Transversal}
Let $T: X \to X$ be a $\cC^{1\plus}$ uniformly expanding Markov map and $\tau: X \to \bR_+$ as above. 
Further suppose that there does not exist some $\theta \in \cC^1(X, \bR)$ such that $\tau = \theta \circ T - \theta + \chi$ where $\chi$ is constant on each partition element.
Then there exists $\ctrans, \gamma >0$ such that, for all $y\in X$, $x_0 \in T^{-n}y$,
\begin{equation}
\label{eq:transversality}
\sum_{\substack{x \in T^{-n}y \\ \cD^n(x)\cK \not\pitchfork \cD^n(x_0)\cK }} J_n(x)  \leq \ctrans e^{-\gamma  n}.
\end{equation}
\end{prop}
\noindent
The major part of the remainder of this subsection is devoted to the proof of this proposition but first we record a consequence of transversality.

\begin{lem}\label{lem:Transversal}
Suppose that $\omega, \varpi  \in \cP_n$, $y\in X$ and that $\cD^{n}(\ell_{\omega}y) \cK \pitchfork \cD^{n}(\ell_{\varpi}y) \cK $. 
Then there exists $L\subset \bR^d$, a $1$-dimensional linear subspace, such that, for all $v\in L$
\[
\abs{D(\tau_{n}\circ \ell_{\omega})(y)v - D(\tau_{n}\circ \ell_{\varpi})(y) v }
 >  \ctaup( \abs{ D\ell_{\omega}(y)v } + \abs{ D\ell_{\varpi}(y)v }).
\]
\end{lem}
\begin{proof}
Let $x_1 = \ell_{\omega}y$, $x_2 = \ell_{\varpi} y$.
That $\cD^n(x_1) \cK \pitchfork  \cD^n(x_2) \cK $ means there exists $L \subset \bR^d$, a line which passes through the origin, such that, when restricted to the two dimensional subspace $L \times \bR \subset \bR^{d+1}$, the image cones  $ \cD^{n}(x_1) \cK $ and $ \cD^{n}(x_2) \cK $ fail to intersect, except at the origin.\footnote{Suppose this were false, then, for all \(L\), restricted to \(L\times \bR\) the cones intersect. If they intersect they intersect in a 1D subspace. We can do this for a set \({\{L_k\}}_{k=1}^{d}\) which are all orthogonal. This constructs a \(d\)-dimensional subspace in the intersection of the images of the cones and this contradicts the assumed transversality.}
Observe that
\[
 \begin{aligned}
  \cD^n(x)\cK \cap L \times \bR 
  & = \left\{  \left(\begin{smallmatrix}
               DT^n(x) \VecA\\
               D\tau_n(x) \VecA + \VecB
              \end{smallmatrix}\right)
              : \abs{\VecB} \leq \ctaup \abs{\VecA},
              DT^n(x) \VecA \in L
  \right\} \\
   & = \left\{  \left(\begin{smallmatrix}
               v \\
               D\tau_n(x)DT^{-n}(x) v + \VecB
              \end{smallmatrix}\right)
              : v \in L,
              \abs{\VecB} \leq \ctaup \abs{DT^{-n}(x) v}
  \right\}.
 \end{aligned}
\]
And consequently  $ \cD^{n}(x_1)\cK \cap \cD^{n}(x_2)\cK \cap L \times \bR = \{0\}$ implies that 
\begin{multline*}
\abs{\left[(D\tau_{n}(x_1)DT^{-1}(x_1) - (D\tau_{n}(x_2)DT^{-1}(x_{2}) \right]v }\\
 >  
 \ctaup \abs{ DT^{-n}(x_1)v }+\ctaup \abs{ DT^{-n}(x_2)v }.  
\end{multline*}
\end{proof}
For all $n\in \bN$, let 
\[
 \phi(n)
 :=
 \sup_{y\in X} \sup_{x_0\in T^{-n}y}
 \sum_{\substack{x \in T^{-n}y \\ \cD^n(x)\cK \not\pitchfork \cD^n(x_0)\cK }} J_n(x).
\]
Let $h_{\nu}$ denote the density of $\nu$ (the $T$-invariant probability measure). 
It is convenient to introduce the quantity
\begin{equation}
\label{eq:DefPhi}
\varphi(n,P,y)
:=
 \sum_{ \substack{x \in T^{-n}(y) \\ \cD^{n}(x)\cK \supset P} } J_{n}(x)  \cdot \frac{h_{\nu}(x)}{h_{\nu}(y)},
\end{equation}
where $P \subset \bR^{d+1}$ is a $d$-dimensional linear subspace.
Let 
\[
 \varphi(n) :=  \sup_{y} \sup_{P} \varphi(n,P,y). 
\]
The benefit of this definition is that $\varphi(n) $ is submultiplicative,
i.e., $\varphi(n+m) \leq \varphi(n) \varphi(m)$ for all $n,m\in \bN$;
and $\varphi(n) \leq 1$ for all $n\in \bN$.
In order to prove Proposition~\ref{prop:Transversal} it suffices to prove the following lemma.
\begin{lem}
\label{lem:AllEquiv}
The following statements are equivalent.
\begin{enumerate}[label={(\roman*)}, font=\normalfont]
\item
$\displaystyle\liminf_{n\to\infty} \phi(n)^{\frac{1}{n}} =1$;
\item
$\displaystyle\lim_{n\to\infty} \varphi(n)^{\frac{1}{n}} =1$;
\item
For all $n\in \bN$ and $y\in X$ there exists a $d$-dimensional linear subspace $Q_n(y) \subset \cK$ such that $\cD^n(x) \cK \supset Q_n(y)$ for all $y$ and for all $x\in T^{-n}y$;
\item
There exists $ \theta \in\cC^1(X,\bR)$ such that $\tau  = \theta\circ T - \theta + \chi$ where $\chi$ is constant on each partition element.
\end{enumerate}
\end{lem}

\begin{proof}[\bfseries Proof of  (i) $\Longrightarrow$ (ii)]
 Let $m_2 \in \bN$, $n= \lceil 2 \frac{\Lambda}{ \lambda} m_2 \rceil $. Since $\Lambda \geq  \lambda$, $n > m_2$. Let $m_1 \in \bN_{+}$ be such that $n=m_1 + m_2$.
 Let $P_n(x_1) := \cD^n(x_1)(\bR^d \times \{0\})$.
 We will first show that $ \cD^n(x_1)\cK \not\pitchfork \cD^n(x_2)\cK$ implies that 
 $\cD^{m_{2}}(T^{m_{1}}x_2)\cK \supset P_n(x_1)$. 
 Observe that
 \[
\cD^n(x)\cK = \left\{\begin{pmatrix} \VecA  \\ D\tau_n(x)DT^{-n}(x)\VecA +\VecB \end{pmatrix} : \VecA \in \bR^d, \VecB \in \bR, \abs{\VecB} \leq \ctaup \abs{DT^{-n}(x) \VecA} \right\}.
\]
That transversality fails means that $P_n(x_1)$ (being contained in $\cD^n(x_1)\cK$) is close to the image cone $\cD^n(x_2)\cK$ by a factor of $\ctaup e^{-\lambda n}$. 
  We also know that $\cD^{m_2}(T^{m_1}x_2)$ is sufficiently bigger than $\cD^n(x_2)\cK$ in the sense that
  \[
  \cD^{m_2}(T^{m_1}x_2) \supset \left\{ \begin{pmatrix} \VecA  \\ \VecB_1 + \VecB_2 \end{pmatrix}  : \begin{pmatrix} \VecA  \\ \VecB_1 \end{pmatrix}  \in \cD^n(x_2)\cK, \abs{\VecB_2} \leq   \ctaup e^{-\lambda n} \abs{\VecA}  \right\}.
  \]
To prove this let $\VecA \in \bR^d$ and $\VecB_0,\VecB_1,\VecB_2 \in \bR$ such that 
$\abs{\VecB_0} \leq \ctaup \abs{DT^{-n}(x_2)\VecA}$,
$\VecB_1 = D\tau_n(x_2) DT^{-n}(x_2)\VecA +\VecB_0$
and
$\abs{\VecB_2} \leq \ctaup e^{-\lambda n}\abs{\VecA}$.
It will suffice to prove that 
\[
 \abs{ (\VecB_1 + \VecB_2 -    D\tau_{m_2}(T^{m_1}x_2) DT^{-m_2}(T^{m_1}x_2))\VecA   }
 \leq \ctaup \abs{ DT^{-m_2}(T^{m_1}x_2))\VecA }.
\]
We estimate
\begin{multline}
  \abs{ (\VecB_1 + \VecB_2 -    D\tau_{m_2}(T^{m_1}x_2) DT^{-m_2}(T^{m_1}x_2))\VecA   }\\
 =
  \abs{ (\VecB_0 + \VecB_1 +    D\tau_{m_1}(x_2) DT^{-m_1}(x_2))DT^{-m_2}(T^{m_1}x_2)\VecA   } \\
  \leq
  \ctaup \left( \tfrac{1}{2} \abs{ DT^{-m_2}(T^{m_1}x_2))\VecA } +  2e^{-\lambda n}\abs{\VecA} \right)
 \\
  \leq
  \ctaup \left(  2e^{-\lambda n}\abs{\VecA} - \tfrac{1}{2} \abs{ DT^{-m_2}(T^{m_1}x_2))\VecA} \right)
  +\ctaup \abs{ DT^{-m_2}(T^{m_1}x_2))\VecA}
\end{multline}
That
$\abs{ DT^{-m_2}(T^{m_1}x_2))\VecA} \geq e^{-\Lambda m_2} \geq e^{-\frac{\lambda}{2} n}$ means 
$  \tfrac{1}{2} \abs{ DT^{-m_2}(T^{m_1}x_2))\VecA} \geq 2e^{-\lambda n}\abs{\VecA}$ for $n$ sufficently large (dependent only on $\lambda$ and $\Lambda$).
 We therefore conclude that $P_n(x_1) \subset \cD^{m_2}(T^{m_1}x_2)$.
Suppose that $x_1 \in T^{-n}y$.
\[
\begin{aligned}
  \sum_{\substack{x_2 \in T^{-n}y \\ \cD^n(x_2)\cK \not\pitchfork \cD^n(x_1)\cK }} J_n(x_2)
  &\leq
  \sum_{\substack{x_2 \in T^{-n}y \\ \cD^{m_2}(T^{m_1}x_2)\cK \supset P_{n}(x_1) }} J_{m_2}(T^{m_1}x_2) J_{m_1}(x_2)\\
    &\leq
  \sum_{\substack{x_3 \in T^{-m_2}y \\ \cD^{m_2}(x_3)\cK \supset P_{n}(x_1) }} J_{m_2}(x_3) 
  \sum_{x_2 \in T^{-m_1}x_3 }  J_{m_1}(x_2).
\end{aligned}
\]
Consequently $\varphi(n) \leq C \phi(m_2(n))$ where $m_2(n) = \lfloor   \frac{n\lambda}{2\Lambda}\rfloor$ and $C = \sup_{x,y} \frac{f_{\nu}(x)}{f_{\nu}(y)}$.
 \end{proof}

\begin{proof}[\bfseries Proof of  (ii) $\Longrightarrow$ (iii)]
First observe that $\displaystyle\lim_{n\to\infty} \varphi(n)^{\frac{1}{n}} =1$ implies $\varphi(n)=1$ for all $n$ since $\varphi(n) $ is submultiplicative and bounded by $1$. 
Consequently the following statement holds:
\begin{quotation}
\textbf{(ii')}
For each $n$ there exists some $y_n\in X$ and some $d$-dimensional linear subspace $Q_n \subset \bR^{d+1}$ such that $  \cD^{n}(x)\cK \supset Q_n$ for every $x \in T^{-n}(y_n)$. 
\end{quotation}
It remains to prove that this above statement implies the following.
\begin{quotation}
\textbf{(iii)}
For all $n\in \bN$ and $y\in X$ there exists a $d$-dimensional linear subspace $Q_n(y) \subset \cK$ such that $\cD^n(x) \cK \supset Q_n(y)$ for all $y$ and for all $x\in T^{-n}y$.
\end{quotation}
We will prove the contrapositive. Suppose the negation of (ii), i.e., 
there exists $n_{0}\in \bN$, $y_0\in X$, $x_{1}, x_{2} \in T^{-n_{0}}(y_0)$ such that $\cD^{n_{0}}(x_{1}) \cK  \cap \cD^{n_{0}}(x_{2}) \cK$ does not contain a $d$-dimensional linear subspace.
Let  $\omega_{1}, \omega_{2} \in \cP_{n_0}$ be such that ($x_1=\ell_{\omega_1} y_0$, $x_2=\ell_{\omega_2} y_0$. These inverses are defined on some neighbourhood $\Delta$ containing $y_0$ and due to the openness related to the cones not intersecting we can assume that 
$\cD^{n_{0}}(\ell_{\omega_1}(y_0)) \cK  \cap \cD^{n_{0}}(\ell_{\omega_2}(y_0)) \cK$ does not contain a $d$-dimensional linear subspace  for all $y \in \Delta$ (shrinking $\Delta$ as required).

There exists $m_{0}\in \bN$ and $\varpi \in \cP_{m_0}$ such that $\ell_\varpi X \subset \Delta$ (using the covering property of $T$).
Observe that, for all $z\in X$,
\[
\cD^{n_0+m_0}(\ell_{\omega_1}(\ell_{\varpi}z)) \cK \subset \cD^{m_0}(\ell_{\varpi} z) \cD^{n_0}(\ell_{\omega_1} y)\cK
\]
where $y=\ell_{\varpi} z$ (and similarly for $\omega_2$).
This means that for all $z\in X$ there exist $x_{1}, x_{2} \in T^{-(m_0+n_0)}(z)$ such that
$\cD^{m_0+n_0}(x_{1}) \cK  \cap\cD^{m_0+n_0}(x_{2})  \cK$ fails to contain a $d$-dimensional linear subspace and consequently contradicts (i').
\end{proof}

\begin{proof}[\bfseries Proof of  (iii) $\Longrightarrow$ (iv)]
Let $(\omega_1, \omega_2,\ldots)$ be a sequence of elements of the partition $\cP$.
For each $n\in \bN$ let $G_n := \ell_{\omega_{n}}\circ\cdots\circ \ell_{\omega_{2}}\circ \ell_{\omega_{1}}$.
Consider 
\begin{equation}
\label{eq:DefOmega}
 D(\tau_n \circ G_n)(x) 
 = \sum_{k=1}^{n} D(\tau \circ \ell_{\omega_{k}})(G_{k-1}x) DG_{k-1}(x)
\end{equation}
and observe that, by \eqref{eq:TauControl} and \eqref{eq:Expanding} this series converges uniformly. 
Moreover this limit is independent of the choice of sequence of inverse branches. This is a consequence of (ii). Observe that 
\[
  \cD^n(x)\cK   
 = \left\{  \left(\begin{smallmatrix}
               v \\
               D\tau_n(x)DT^{-n}(x) v + \VecB
              \end{smallmatrix}\right)
              : v \in \bR^d,
              \abs{\VecB} \leq \ctaup \abs{DT^{-n}(x) v}
  \right\}.
\]
Therefore, for all $n$, $y\in X$, then 
\[
\norm{  D\tau_n(x_1)DT^{-n}(x_1) v -  D\tau_n(x_2)DT^{-n}(x_2) v   }
\leq  2 \ctaup \norm{v} \lambda^{-n}
\]
for all $x_1, x_2 \in T^{-n}y$.

Consequently we can denote by $\Omega(x)$ the limit of \eqref{eq:DefOmega}.
It holds that, for all $\omega \in \cP$, 
\[
  \Omega(x) = D(\tau\circ \ell_{\omega})(x) + \Omega(\ell_{\omega} x) D\ell_{\omega}(x).
\]
Fix $x_0 \in X$. The series of functions $\sum_{k=1}^\infty (\tau\circ G_n - \tau\circ G_n(x_0))$ is summable in $\cC^1$.
Denote this sum by $\theta$.
By construction $\Omega(x) = D\theta(x)$. 
Consequently $D(\tau + \theta - \theta\circ T)=0$.
\end{proof}

\begin{proof}[\bfseries Proof of  (iv) $\Longrightarrow$ (i)]
 Let
 \[
  Q(x) := \left\{  \left(\begin{smallmatrix}
               \VecA \\
               D\theta(x) \VecA
              \end{smallmatrix}\right)
              : \VecA \in \bR^d
         \right\}. 
 \]
 Observe that $Q(x) \subset \cK$.
 Since $D\tau_n(x) = D\theta(T^n x) DT^n(x) - D\theta$,
  \[
  \begin{aligned}
     \cD(x)^n Q(x) &= \left\{  \left(\begin{smallmatrix}
               DT^n(x) & 0 \\
               D\tau_n(x) & 1 
              \end{smallmatrix}\right) \left(\begin{smallmatrix}
               \VecA \\
               D\theta(x) \VecA
              \end{smallmatrix}\right)
              : \VecA \in \bR^d
         \right\}\\
   &=
         \left\{  \left(\begin{smallmatrix}
               DT^n(x) \VecA \\
               D\theta(T^n x)DT^n(x) \VecA 
              \end{smallmatrix}\right) 
              : \VecA \in \bR^d
         \right\}
         = Q(T^n x).
   \end{aligned}
\]
This means that for all $y\in X$ then $\cD(x)^n \cK \supset Q(y)$ for all $x\in T^{-n}y$.
\end{proof}

\subsection{Oscillatory Cancellation}
In this subsection we take advantage of the geometric property established above and estimate the resultant cancellations. 
The following estimate concerns the case when $f$ is more or less constant on a scale of $\abs{b}^{-1}$.
The argument will depend on the following choice of constants (chosen conveniently but not optimally)
\[
 \beta_1 := \frac{2}{\lambda},
 \quad
 \beta_2 := \frac{\alpha}{8\Lambda},
 \quad
 q := \frac{\alpha \lambda}{2}.
\]
Let $n_1 = \lfloor\beta_1 \ln\abs{b}\rfloor$, $n_2 = \lfloor \beta_2 \ln\abs{b} \rfloor$ and $n := n_1 + n_2$, $\beta := \beta_1 + \beta_2$. The first $n_1$ iterates will be so that the dynamics evenly spreads the function $f$ across the space $X$. Then $n_2$ iterates will be to see the oscillatory cancellations.
The assumptions of the following proposition are identical to the assumptions of Proposition~\ref{prop:Transversal}.
\begin{prop}
 \label{prop:Dolgopyat}
 Let $T: X \to X$ be a $\cC^{1\plus}$ uniformly expanding Markov map and $\tau: X \to \bR_+$ as above. 
Further suppose that there does not exist some $\theta \in \cC^1(X, \bR)$ such that $\tau = \theta \circ T - \theta + \chi$ where $\chi$ is constant on each partition element.

 Then there exists $\xi > 0$, $b_0>1$, \(\sigma>0\) such that, for all  $z=a+ib$, $a\in (-\sigma, \sigma)$, $\abs{b} > b_0$, $n = \lfloor\beta \ln \abs{b}\rfloor$, for all $f\in \cC^{\alpha}(X)$ satisfying $  \abs{f}_{\cC^{\alpha}(X)} \leq  e^{q n} \abs{b}^{\alpha}  \norm{f}_{L^{\infty}(X)} $ we have
 \[
 \norm{\cL^{n}_{z} f}_{L^1(X)} \leq e^{-\xi n} \norm{f}_{L^\infty(X)}.
 \] 
\end{prop}
\noindent
The proof follows after several lemmas.

It is convenient to localize in space using a partition of unity. Using the assumption that the box-counting dimension of the boundary is strictly smaller than the ambient dimension we have the following partition of unity.
\begin{lem}
\label{lem:PartOfUnity}
There exist $\cPartUnityA, \cPartUnityB, r_0>0$, \(\cdim \in [0,d)\) such that, for all $r\in (0,r_0)$ there exists a set of points ${\{x_p\}}_{p=1}^{N_r}$ and a $\cC^1$ partition of unity ${\{\rho_p\}}_{p=1}^{N_r}$ of $X$ (i.e., $\sum_p \rho_p(x) = 1$ for all \(x\in X\), $\rho_p \in \cC^1(X, [0,1])$) with the following properties.
\begin{itemize}
 \item $N_r \leq \cPartUnityA r^{-d}$;
\end{itemize}
For each $p$,
\begin{itemize}
 \item $\rho_p(x) = 1$ for all $x\in B(x_p,  r)$;
 \item $\supp(\rho_p) \subset  B(x_p, \cPartUnityA r)$;
 \item $\norm{\rho_p}_{\cC^1} \leq \cPartUnityA r^{-1}$;
\end{itemize}
And, letting \(\cR_\partial  := \left\{p : B(x_p, \cPartUnityA r) \cap \partial \omega \neq \emptyset \text{ for some $\omega \in \cP$}\right\}\),
\begin{itemize}
 \item  $\# \cR_\partial \leq \cPartUnityB r^{-\cdim}$.
\end{itemize}
\end{lem}
\noindent
The construction of such a partition of unity and the proof of the above estimates are given in Appendix~\ref{app:DimOfBoundary}.

At each different point of $X$ we have a direction in which we see cancellations. One  major use of the partition of unity is to consider the direction as locally constant.
We choose $r=r(b) = \abs{b}^{-\frac{1}{2}}$. 
Take $f\in \cC^{\alpha}(X)$.
Using Jensen's inequality 
\[
 \begin{aligned}
  \norm{\cL^n_z f}_{L^1(X)}
  &= 
  \int_X \abs{\sum_{\omega \in \cP_{n}} (J_n \cdot f \cdot e^{-z\tau_n})\circ \ell_{\omega}(x) \cdot \one_{T^{n}\omega}(x) } \ dx\\
  & =
   \sum_{p=1}^{N_r} \int \abs{\sum_{ \omega\in \cP_n} \rho_p \cdot (J_n \cdot f \cdot  e^{-z\tau_n})\circ \ell_{\omega}(x) \cdot \one_{T^{n}\omega}(x)} \ dx\\
  & \leq
   \left( \sum_{p\not\in \cR_{\partial}} \sum_{\omega,\varpi \in \cP_n}\abs{ \int_{T^{n}\omega \cap T^{n}\varpi} (\rho_p \cdot K\circ \ell_{\omega} \cdot K\circ\ell_{\varpi} \cdot e^{ib\theta_{\omega,\varpi}})(x)\ dx } \right)^{\frac{1}{2}}\\
   & \ \ + e^{\sigma n}  \norm{f}_{L^{\infty}}  \sum_{ \omega\in \cP_n}  \norm{J_{n}}_{L^{\infty}(\omega)}    \sum_{p\in \cR_{\partial}}  \int_{T^{n}\omega} \rho_p(x)    \ dx 
 \end{aligned}
\]
where $K := (J_n\cdot f \cdot e^{-a\tau_n})$ and
$\theta_{\omega,\varpi}:= \tau_n\circ \ell_{\omega} - \tau_n\circ \ell_{\varpi}$.  
 Using Lemma~\ref{lem:DistortionB} and Lemma~\ref{lem:PartOfUnity}, the final term of the above is bounded by
\begin{equation}
\label{eq:Boundary}
\cPartUnityB 2^{d} r^{d-\cdim}  \cbalance   e^{\sigma n} \norm{f}_{L^{\infty}(X)}
\leq
\cPartUnityB 2^{d}  \cbalance e^{-(\frac{d-\cdim}{2\beta}-\sigma)n} \norm{f}_{L^{\infty}(X)}.  
\end{equation}
It remains to estimate the other term.
We estimate separately the set  
\[
\cQ_{n,p,\omega}:=\{ \varpi \in \cP_{n}:  \cD^{n_2}(T^{n_1}\ell_{\varpi} x_p)\not\pitchfork  \cD^{n_2}(T^{n_1}\ell_{\omega} x_p) \}
\]
 and the set of $\varpi$ where this is not the case. In the second case we see oscillatory cancellations.

\begin{lem}
\label{lem:RegK}
There exists $\cRegK>0$ such that 
\begin{multline*}
\abs{ K\circ \ell_{\omega}(x) - K\circ \ell_{\omega}(y)}\\
\leq 
 e^{\sigma n} \norm{J_{n}}_{L^{\infty}(\omega)} (\cRegK \norm{f}_{L^{\infty}(\omega)} + \abs{f}_{\cC^{\alpha}(\omega)}e^{-\alpha \lambda n}  ) \dist{x}{y}^\alpha
\end{multline*}
{for all $n\in \bN$, $\omega \in \cP_n$, $x,y \in T^{n}\omega$.}
\end{lem}
\begin{proof}
Since $  {K\circ \ell_{\omega}(x)}
    =
   {(J_n\cdot f \cdot e^{-a\tau_n})\circ \ell_{\omega}(x)}$,
  for all $x,y \in T^{n}\omega$, 
 \[
 \begin{aligned}
 K\circ \ell_{\omega}(x) - K\circ \ell_{\omega}(y)
   & = ( e^{-a\tau_n(\ell_{\omega}x)} - e^{-a\tau_n(\ell_{\omega}y)}) f(\ell_{\omega}x) \cdot J_n(\ell_{\omega}x)\\
  & \ + e^{-a\tau_n(\ell_{\omega}y)}f(\ell_{\omega}y) (J_n(\ell_{\omega}x)-J_n(\ell_{\omega}y))\\
  & \ + e^{-a\tau_n(\ell_{\omega}y)}(f(\ell_{\omega}x)-f(\ell_{\omega}y)) \cdot J_n(\ell_{\omega}x).
     \end{aligned}
\]
Using the estimates of   Lemma~\ref{lem:TauPrime}, Lemma~\ref{lem:DistortionA} and \eqref{eq:Expanding},
\begin{multline*}
\abs{ K\circ \ell_{\omega}(x) - K\circ \ell_{\omega}(y)}\\
\leq 
e^{\sigma n} \norm{J_{n}}_{L^{\infty}(\omega)} 
\left( ({\sigma} \frac{\ctaup}{2}   +  \cdistorn)\norm{f}_{L^{\infty}(\omega)}    +  \abs{f}_{\cC^{\alpha}(\omega)}e^{-\alpha \lambda n}  \right)
\dist{x}{y}^\alpha
\end{multline*}
The lemma follows from choosing $\cRegK := \cdistorn  + \sigma \frac{\ctaup}{2}$.
\end{proof}

\begin{lem}
\label{lem:RegTheta}
There exists $\cRegTheta > 0$ such that, for all $n\in \bN$, $\omega,\varpi \in \cP_n$,
 \[
 \norm{D\theta_{\omega,\varpi}}_{\cC^{\alpha}} \leq \cRegTheta.
 \]
\end{lem}
\begin{proof}
 $
  D(\tau_n \circ \ell_{\omega})(x)
  =
  \sum_{k=0}^{n-1} D\tau(h_k x) Dh_k(x)
 $
 where $h_k := T^k \circ \ell_{\omega}$.
So
\[
\begin{aligned}
 \norm{D(\tau_n \circ g)(x) - D(\tau_n \circ g)(y)}
 &\leq 
 \sum_{k=0}^{n-1} \norm{D\tau}_{\cC^{\alpha}} \dist{h_kx}{h_ky}^\alpha\\
 &\leq
 \norm{D\tau}_{\cC^{\alpha}}   \sum_{k=0}^{n-1} e^{-\lambda n\alpha}   \dist{x}{y}^\alpha.
 \end{aligned}
\]
And so 
$ \norm{D\theta_{g,h}}_{\cC^{\alpha}} \leq 2  \norm{D\tau}_{\cC^{\alpha}}    \sum_{k=0}^{\infty} e^{-\lambda n\alpha}  $.
\end{proof}

\begin{lem}
Suppose the setting of Proposition~\ref{prop:Dolgopyat}. There exists $\cOther >0$ such that
\label{lem:Others}
\begin{multline}
   \left(\sum_{p\not\in \cR_{\partial}} \sum_{\omega \in \cP_n} \sum_{\varpi\in \cQ_{n,p,\omega}}\abs{ \int (\rho_p \cdot K\circ \ell_{\omega} \cdot K\circ \ell_{\varpi} \cdot e^{ib\theta_{\omega,\varpi}})(x)\ dx }\right)^{\frac{1}{2}}\\
  \leq 
  \cOther e^{-(\frac{\gamma \beta_2}{2 \beta} - \sigma  ) n} \norm{f}_{L^{\infty}}.
\end{multline}
\end{lem}
\begin{proof}
Fixing for the moment $p\notin \cR_{\partial}$ and $\omega \in \cP_{n}$ we want to perform the sum over $\varpi$. 
\begin{multline}
\label{eq:OtherA}
 \sum_{\varpi\in \cQ_{n,p,\omega}}\abs{ \int_{T^{n}\omega \cap T^{n}\varpi} (\rho_p \cdot K\circ \ell_{\varpi}\cdot K\circ \ell_{\omega} \cdot e^{ib\theta_{\omega,\varpi}})(x)\ dx } 
\\
\leq
  \left(\sum_{\varpi\in \cQ_{n,p,\omega}}   \norm{J_{n}}_{L^{\infty}(\varpi)} \right) e^{2\sigma n} \norm{J_{n}}_{L^{\infty}(\omega)}\norm{f}_{L^{\infty}}^2 \norm{\rho_p}_{L^{1}}.
\end{multline}
Observe that
\[
 \sum_{\varpi\in \cQ_{n,p,\omega}}   \norm{J_{n}}_{L^{\infty}(\varpi)}
 \leq 
 \left(  \sum_{\varpi_1\in \cP_{n_1}}   \norm{J_{n}}_{L^{\infty}(\varpi_1)} \right)
 \left(  \sum_{\varpi_2}   \norm{J_{n}}_{L^{\infty}(\varpi_2)} \right)
\]
where the second sum is over the set of $\varpi_2 \in \cP_{n_2}$ which satisfy
\[
\cD^{n_2}(T^{n_1}\ell_{\varpi_{2}} x_p)\pitchfork  \cD^{n_2}(T^{n_1}\ell_{\omega} x_p).
\]
Consequently, applying the estimate of Proposition~\ref{prop:Transversal}, the term in \eqref{eq:OtherA} is bounded by 
\[
 \ctrans \cbalance e^{-\gamma  n_2}  e^{2\sigma n} \norm{J_{n}}_{L^{\infty}(\omega)}\norm{f}_{L^{\infty}}^2 \norm{\rho_p}_{L^{1}}. 
\]
Using again Lemma~\ref{lem:DistortionB}, 
$\sum_{\omega \in \cP_n} \norm{J_{n}}_{L^{\infty}(\omega)}  \leq \cbalance$ 
and we sum over $p$.
\end{proof}

Now we turn our attention to the $\varpi \in \cP_n$ where we observe oscillatory cancellations.
The crucial technical part of the estimate is the following oscillatory integral bound.
\begin{lem}
\label{lem:OscInt}
Suppose that $J \subset [0,1]$ is an interval, $k\in \cC^\alpha(J)$, $\theta \in \cC^{1+\alpha}(J)$, $\abs{\theta'} \geq \kappa >0$, $\abs{b} > 1$, $k\in \cC^\alpha(J)$.
Then
\[
\abs{\int_{J} e^{ib\theta(x)} k(x) \ dx}
\leq
\frac{C}{\kappa^2 \abs{b}^\alpha} \norm{k}_{\cC^{\alpha}(J)}.
\]
where $C = (\norm{\theta'}_{L^\infty(X)} + 6)(1 + \abs{\theta'}_{\cC^\alpha(X)})$.
\end{lem}
\begin{proof}
We assume that $b>1$, the other case being identical.
We also assume without loss of generality that ${\theta'} \geq \kappa$ otherwise we can exchange $-\theta$ for $\theta$.
 Since $\frac{k}{\theta'}$ is $\alpha$-H\"older there exists\footnote{
 Take a molifier $\rho \in \cC^1(\bR,[0,1])$ such that $\supp(\rho)\subset (-1,1)$, $\int \rho = 1$, $\int \abs{\rho'} \leq 2$. Define
 \[
 g_b(x) := \int \rho_b(x-y) \tfrac{k}{\theta'}(y) \ dy
 \]
 where $\rho_b(z) := b \rho(bz)$.
 Observe that 
 $
   g_b(x) - \tfrac{k}{\theta'}(x) 
   = \int  \rho_b(x-y) \left[\tfrac{k}{\theta'}(y) - \tfrac{k}{\theta'}(x)\right] \ dy
 $, that
 $
 g_b'(x) 
 = \int  \rho_b'(x-y) \left[\tfrac{k}{\theta'}(y) - \tfrac{k}{\theta'}(x)\right] \ dy
 $,
 that 
 $\int \abs{\rho_b} = 1$
 and that
 $\int \abs{\rho_b'} \leq 2b$.
} 
$g_b \in \cC^1(J, \bR)$ such that
 \[
\norm{g_b - \tfrac{k}{\theta'}}_{L^\infty} \leq b^{-\alpha} \abs{\tfrac{k}{\theta'}}_{\cC^\alpha},\quad
\norm{g_b'}_{L^\infty} \leq 2 b^{1-\alpha} \abs{\tfrac{k}{\theta'}}_{\cC^\alpha}.
 \]
 Changing variables, $y=\theta(x)$,   
 \[
 \begin{aligned}
    \int_{J}  k(x) \cdot  e^{ib \theta (x) } \ dx
   &= \int_{\theta(J)} \frac{k}{\theta'}\circ \theta^{-1}(y) e^{ib y } \ dy\\
   &=\int_{\theta(J)} g_b\circ \theta^{-1}(y) e^{ib y } \ dy
   \\
   & \quad \quad \quad \quad
   + \int_{\theta(J)} \left( \tfrac{k}{\theta'} - g_b \right) \circ \theta^{-1}(y) e^{ib y } \ dy.
    \end{aligned}
   \]
   Observe that the final term is equal to
   \(   \int_{J} ( \tfrac{k}{\theta'} - g_b )(x) e^{ib \theta(x) } \theta'(x) \ dx\).
Integrating by parts the penultimate term,
 \[
 \begin{aligned}
   \int_{\theta(J)} g_b\circ \theta^{-1}(y) e^{ib y } \ dy
  & = 
 - \frac{i}{b}  \left[  g_b\circ \theta^{-1}(y)   e^{ib y }  \right]_{\theta(J)} 
 \\
 & \quad \quad \quad \quad + \frac{i}{b} \int_{\theta(J)}   \frac{g_b'}{\theta'}  \circ \theta^{-1}(y) e^{ib y } \ dy\\
 & = 
-  \frac{i}{b}   \left[  g_b  e^{ib \theta }  \right]_{J}    
  + \frac{i}{b} \int_{J} g_b'(x) e^{ib \theta(x) } \ dx.
    \end{aligned}
 \]
Combining these estimates
\[
\begin{aligned}
 \abs{\int_{J} e^{ib\theta(x)} k(x) \ dx}
&\leq  
\abs{ \int_{J} ( \tfrac{k}{\theta'} - g_b )(x) e^{ib \theta(x) } \theta'(x) \ dx }
 + \abs{\frac{1}{b}   \left[  g_b  e^{ib \theta }  \right]_{J}}\\
& \quad +\abs{\frac{1}{b} \int_{J} g_b'(x) e^{ib \theta(x) } \ dx}\\
&\leq
\left( \frac{\norm{\theta'}_\infty \abs{J}}{b^\alpha} + \frac{2}{b^{1+\alpha}} + \frac{2 \abs{J}}{b^\alpha}  \right) \abs{\frac{k}{\theta'}}_\alpha 
+ \frac{2\norm{k}_\infty}{b \kappa}.
\end{aligned}
\]
To finish we observe that
\[
\begin{aligned}
 \abs{\frac{k}{\theta'}(x) - \frac{k}{\theta'}(y) }
 &=
  \abs{\frac{k(x) - k(y)}{\theta'(x)} + \frac{k(y) (\theta'(y) - \theta'(x))}{\theta'(x) \theta'(y)}}\\
  &\leq
  \left( \frac{\abs{k}_\alpha}{\kappa} +  \frac{\norm{k}_\infty \abs{\theta'}_\alpha}{\kappa^2} \right)\abs{x-y}^\alpha.      
  \end{aligned}\qedhere
\]
\end{proof}

\begin{lem}
Suppose the setting of Proposition~\ref{prop:Dolgopyat}. There exists $\cCancel >0$ such that
\label{lem:Cancel}
\begin{multline}
   \left(\sum_{p\not\in \cR_{\partial}} \sum_{\omega \in \cP_n} \sum_{\varpi\in \cP_{n}\setminus \cQ_{n,p,\omega}}\abs{ \int (\rho_p \cdot K\circ \ell_{\omega} \cdot K\circ\ell_{\varpi} \cdot e^{ib\theta_{\omega,\varpi}})(x)\ dx }\right)^{\frac{1}{2}}\\
  \leq 
\cCancel   \abs{b}^{-\frac{\alpha}{4} } 
e^{(\frac{\Lambda\beta_2}{\beta} +\sigma) n} \norm{f}_{L^{\infty}}
 \leq 
\cCancel  e^{-(\frac{\alpha}{8\beta} -\sigma) n} \norm{f}_{L^{\infty}}.
\end{multline}
\end{lem}
\begin{proof}
Fixing for the moment $p$ and $\omega$ we want to perform the sum over $\varpi$. I.e., we estimate
\[
 \sum_{\varpi\in \cP_{n}\setminus \cQ_{n,p,\omega}}\abs{ \int_{T^{n}\omega \cap T^{n}\varpi} (\rho_p \cdot K\circ \ell_{\omega} \cdot K\circ\ell_{\varpi} \cdot e^{ib\theta_{\omega,\varpi}})(x)\ dx }
 \]
Since 
$\cD^{n_2}(T^{n_1}\ell_{\varpi} x_p) \cK \pitchfork  \cD^{n_2}(T^{n_1}\ell_{\omega} x_p)  \cK$
 there exists (Lemma~\ref{lem:Transversal}) a $1$-dimensional linear subspace $L \subset \bR^{d}$ (which depends on \(\varpi\) and \(\omega\)) such that, for all $v\in L$,
\[
\abs{D(\tau_{n_2}\circ T^{n_1} \circ  \ell_{\varpi})(x_{p})v - D(\tau_{n_2}\circ  T^{n_1} \circ  \ell_{\omega})(x_{p}) v }
 >  \ctaup \abs{ D( T^{n_1} \circ \ell_{\varpi})(x_{p})v }.
\]
(We could also write another term on the right hand side of the above but this worse estimate suffices for what follows.) By Lemma~\ref{lem:TauPrime}
\[
\abs{D(\tau_{n_1} \circ  \ell_{\varpi})(x_{p})v  }
 \leq   \tfrac{\ctaup}{2} \abs{ D( T^{n_1} \circ \ell_{\varpi})(x_{p})v }.
\]
Consequently
\begin{multline*}
\abs{D(\tau_{n}\circ  \ell_{\varpi})(x_{p})v - D(\tau_{n}\circ  \ell_{\omega})(x_{p}) v }
\\
 >  \tfrac{\ctaup}{2} (\abs{ D( T^{n_1} \circ \ell_{\varpi})(x_{p})v } + \abs{ D( T^{n_1} \circ \ell_{\varpi})(x_{p})v }).
\end{multline*}
Rotating and translating the axis, we choose an orthogonal coordinate system $(y_1,y_2,\ldots, y_d)$ such that $y_1$ corresponds to $L$ and such that $x_{p} = (0,\ldots,0)$. 
We have
\[
  \abs{\tfrac{\partial \theta_{\omega,\varpi}}{\partial y_1}}(0,\ldots, 0)
 =
 \abs{\tfrac{\partial(\tau_{n}\circ \ell_{\varpi})}{\partial y_1} - \tfrac{\partial(\tau_{n}\circ \ell_{\omega})}{\partial y_1} }(0,\ldots, 0)\\
 \geq 
 {\ctaup}e^{-\Lambda n_2}.
\]
Since $r>0$ is sufficiently small the transversality holds along this direction for the entire ball (using Lemma~\ref{lem:RegTheta}). 
In order to show this we will show that $\cRegTheta r(b)^{\alpha} \leq  \frac{\ctaup}{2}e^{-\Lambda n_2} $ since $ \norm{D\theta_{\omega,\varpi}}_{\cC^\alpha} \leq \cRegTheta$.
This is equivalent to requiring $\exp(-[\frac{\alpha}{2\beta_2}-\Lambda]n_2 ) \leq \frac{\ctaup}{2\cRegTheta}$ which holds for $\abs{b}$ sufficiently large since $\beta_2$ was chosen such that $\beta_2 \leq \frac{\alpha}{2\Lambda}$.
Here $b_{0}$ is chosen sufficiently large to guarantee that \(\abs{b}\) is large enough to satisfy the above condition.
We have
\[
 \abs{\tfrac{\partial \theta_{\omega,\varpi}}{\partial y_1}}(y_1,\ldots, y_d)
 =
 \abs{\tfrac{\partial(\tau_{n}\circ \ell_{\varpi})}{\partial y_1} - \tfrac{\partial(\tau_{n}\circ \ell_{\omega})}{\partial y_1} }(y_1,\ldots, y_d)\\
 \geq 
 \tfrac{\ctaup}{2}e^{-\Lambda n_2}
\]
for all $(y_1,\ldots, y_d) \in B_{r_b}(0)$.
To proceed we must estimate the H\"older norm of $\rho_p \cdot K\circ \ell_{\omega} \cdot K\circ\ell_{\varpi}$.
By Lemma~\ref{lem:RegK}, since we assume that $  \abs{f}_{\cC^{\alpha}(\omega)}\leq e^{(q+\frac{\alpha}{ \beta} )n}   \norm{f}_{L^{\infty}(\omega)} $ in Proposition~\ref{prop:Dolgopyat} and $q+\frac{\alpha}{ \beta} \leq \alpha(\frac{\lambda}{2} + \frac{1}{\beta_1}) = \alpha \lambda$ (for some $C>0$),
\[
\abs{ K\circ \ell_{\omega}(x) - K\circ \ell_{\omega}(y)}
\leq 
C e^{\sigma n} \norm{J_{n}}_{L^{\infty}(\omega)} \norm{f}_{L^{\infty}(\omega)} \dist{x}{y}^\alpha
\]
Consequently, using Lemma~\ref{lem:RegK} and Lemma~\ref{lem:PartOfUnity},
\[
\abs{\rho_p \cdot K\circ \ell_{\omega} \cdot K\circ\ell_{\varpi}}_{\cC^{\alpha}(T^{n}\omega)}
\leq
C \left( 1  + r^{-\alpha}\right)  e^{\sigma n} \norm{J_{n}}_{L^{\infty}(\omega)} \norm{f}_{L^{\infty}(\omega)}.
\]
Using the estimate of Lemma~\ref{lem:OscInt}, for $(y_2,\ldots,y_d)$ fixed,
\begin{multline*}
 \abs{\int_{-r}^{r}  (\rho_p \cdot K\circ \ell_{\omega} \cdot K\circ\ell_{\varpi} \cdot e^{ib\theta_{\omega,\varpi}})(y_1,\ldots,y_d)\ dy_1}
 \\
 \leq
C r^{-\alpha} e^{2 \Lambda n_2}  \abs{b}^{-\alpha} 
e^{2\sigma n} \norm{J_{n}}_{L^{\infty}(\omega)} \norm{J_{n}}_{L^{\infty}(\varpi)}\norm{f}_{L^{\infty}}^2.
\end{multline*}
If $d=1$ we are done, otherwise we integrate over the other directions. We also recall that  $r= \abs{b}^{-\frac{1}{2}}$.
\begin{multline*}
\abs{ \int_{T^{n}\omega \cap T^{n}\varpi} (\rho_p \cdot K\circ \ell_{\omega} \cdot K\circ\ell_{\varpi} \cdot e^{ib\theta_{\omega,\varpi}})(x)\ dx }\\
 \leq
C   \abs{b}^{-\frac{\alpha}{2}} 
e^{2\Lambda n_2 + 2\sigma n} \norm{J_{n}}_{L^{\infty}(\omega)} \norm{J_{n}}_{L^{\infty}(\varpi)}\norm{f}_{L^{\infty}}^2.
\end{multline*}
Using Lemma~\ref{lem:DistortionB} we sum over $\omega$ and $\varpi$ to obtain the estimate.
\end{proof}

\begin{proof}[Proof of Proposition \ref{prop:Dolgopyat}]
The estimates from \eqref{eq:Boundary}, Lemma~\ref{lem:Cancel} and Lemma~\ref{lem:Others} imply that, for some $C>0$,
\[
\norm{\cL^{n}_z f}_{L^1(X)}
\leq
C\left(
 e^{-(\frac{\gamma \beta_2}{2 \beta} - \sigma  ) n}
+
e^{-(\frac{\alpha}{8\beta}-\sigma) n}
+
e^{-(\frac{d-\cdim}{2\beta}-\sigma)n}
\right)
\norm{f}_{L^{\infty}(X)}.
\]
Here we insure that $\sigma>0$ is sufficiently small, dependent only on the system.
\end{proof}

\begin{prop}
\label{prop:FuncResult}
 Let $T: X \to X$ be a $\cC^{1\plus}$ uniformly expanding Markov map and $\tau: X \to \bR_+$ as above. 
Further suppose that there does not exist some $\theta \in \cC^1(X, \bR)$ such that $\tau = \theta \circ T - \theta + \chi$ where $\chi$ is constant on each partition element.

Then there exists $\zeta, b_0, B >0$ such that, for all $z=a+ib$, $a \geq -\sigma$, $\abs{b} \geq b_0$, $n\geq  B \ln  \abs{b} $
\[
\norm{\cL^n_z}_{(b)} \leq  e^{-\zeta n}.
\]
\end{prop}
\begin{proof}
We first estimate \( \norm{\cL^n_z}_{(b)}  \) for $n = \beta \ln \abs{b}$.
We will estimate this quantity in two separate cases. Firstly we consider the case when
\[
\norm{f}_{L^{\infty}(X)} \leq e^{-q n}\norm{f}_{(b)}.
\]
We apply Lemma~\ref{lem:AdaptedLY}:
\[
 \norm{\cL^n_z f}_{(b)} 
 \leq 
 \cly e^{\sigma n} \left( e^{-\lambda n} \norm{f}_{(b)}
 +   \norm{f}_{L^\infty(X)}\right)
 \leq
  C e^{\sigma n} (e^{-\lambda n} +   e^{-q n} ) \norm{f}_{(b)}
\]
It remains to consider the case when $\norm{f}_{L^{\infty}} \geq e^{-q n}\norm{f}_{(b)}$.
This means that  $ \abs{f}_{\cC^\alpha(X)} \leq  e^{q n} (1+\abs{b}^{\alpha} ) \norm{f}_{L^{\infty}(X)} $.
The interpolation result of Lemma~\ref{lem:interpol} means  
that there exists $C, \epsilon_0>0$ such that, for any $\epsilon\in (0,\epsilon_0)$, 
\begin{equation}
 \label{eq:EstimateSup}
 \norm{f}_{L^\infty(X)} \leq
 C \epsilon^{-d} \norm{f}_{L^1(X)}
 +\epsilon^\alpha \abs{f}_{\cC^{\alpha}(X)}.
\end{equation}
Here we choose $\epsilon = e^{-\frac{\xi}{2d}n}$.
Applying Lemma~\ref{lem:AdaptedLY} twice
\[
 \norm{\cL^{2n}_z f}_{(b)} 
 \leq \cly e^{2 \sigma n}  e^{-\lambda n} \norm{f}_{(b)}
 + \cly e^{\sigma n}  \norm{\cL^{n}_z f}_{L^{\infty}(X)}.
\]
Using also the above estimate \eqref{eq:EstimateSup}
\[
 \norm{\cL^{2n}_z f}_{(b)} 
 \leq \left(\cly  e^{-(\lambda-2 \sigma) n}
 +  e^{-\frac{\alpha \xi}{2d}n} \right) \norm{f}_{(b)}
 + \cly e^{(\sigma + \frac{\xi}{2})n}  \norm{\cL^{n}_z f}_{L^{1}(X)}.
\]
The estimate of Proposition~\ref{prop:Dolgopyat} means that
\[
 \norm{\cL^{2n}_z f}_{(b)} 
 \leq \left(\cly  e^{-(\lambda-2 \sigma) n}
 +  e^{-\frac{\alpha \xi}{2d}n} \right) \norm{f}_{(b)}
 + \cly e^{-(\frac{\xi}{2}-\sigma)n}  \norm{f}_{(b)}.
\]
Again we ensure that $\sigma>0$ is sufficiently small. 
We have obtained the estimate $\norm{\cL^n_z}_{(b)} \leq  e^{-\zeta n}$ when $n = \lfloor \beta \ln \abs{b} \rfloor$. Iterating this estimate and choosing $B>0$ sufficiently large concludes the proof.
\end{proof}

\subsection{Rate of Mixing}
It remains to complete the proof of Theorem~\ref{thm:ExpSemi}.
In the present setting, in particular that the twisted transfer operators satisfy a Lasota-Yorke style estimate (Lemma~\ref{lem:AdaptedLY}),  the required conclusion of exponential mixing follows in an established fashion (for example \cite[\S 2.7]{AM15} or \cite[\S 7.5]{AGY06}) from the estimate of Proposition~\ref{prop:FuncResult}. In the first cited reference the $\cC^1$ norm is used whilst in our case the $\cC^\alpha$ norm is used but the same argument holds since it depends on the spectral properties of the twisted transfer operator and the norm estimate (Proposition~\ref{prop:FuncResult}) and these are identical in the present case.
In the second cited reference the $\cC^\alpha$ norm is used but for functions of the interval and not the higher dimensional situation of the present work. Again the argument presented there depends only on the spectral properties of the operator and so holds also in this setting. 

For the convenience of the reader we here summarise the general argument which was cited above, at each stage the relevant paragraph in one of the references is detailed. 
The main part of the argument is to observe that the Laplace transform of the correlation function can be written in terms of a sum of twisted transfer operators \cite[Proposition A.3]{AM15}. 
The Laplace transform of the correlation is then shown to admit an analytic extension to a neighbourhood of each point \(z = ib\). For \(b\neq 0\) this is because the existence of poles on the imaginary axis would contradict mixing since they form groups and for \(z=0\) this uses that the problem reduces to the case when one of the observables is zero average~\cite[Lemma 2.22]{AM15}. This part of the argument uses the quasi-compactness of the twisted transfer operators.
When \(\abs{b}\) is large the main functional-analytic estimate (Proposition~\ref{prop:FuncResult}) is used to imply an analytic extension of uniform size \cite[Lemma 2.23]{AM15}.
The above is done in a way which is independent on the choice of observables.
Combining the above gives an analytic extension to the correlation function to a strip about the imaginary axis. The result of exponential mixing then follows from a Paley-Weiner type estimate \cite[\S 2.7]{AM15}.

\appendix

\section{The Boundary of Markov Partitions}
\label{app:markov}

In the early 1970s, Bowen~\cite{Bowen73} and Ratner~\cite{Ratner73} showed that it is possible to construct Markov partitions for Anosov 
flows. However it is known~\cite{Bowen78} that the regularity of the boundary of these partitions is normally rather bad. 
This is unfortunate for our present purposes since we need some degree of regularity of the unstable part of the Markov construction in order to complete our argument. 
Ratner~\cite{Ratner73} showed that the boundary of the Markov partition has Lebesgue measure zero but this is not quite sufficient for our purposes. Fortunately, as shown by Horita \& Viana~\cite[Proposition 3.5]{HV05} we also have estimates for the box-counting dimension\footnote{In general the upper box-counting dimension may differ from the lower box-counting dimension. Throughout this text our only interest is in an upper bound for the upper box-counting dimension and for conciseness we consistently omit explicit mention of this detail. Note that in the reference cited~\cite{HV05} for the dimension result the term \emph{limit capacity} is used for the same concept.} of the boundary. Section~\ref{app:DimOfBoundary} is devoted to reviewing this topic and the information on the dimension of the boundary is a key point in constructing the partition of unity of Lemma~\ref{lem:PartOfUnity}. 
Section~\ref{sec:john} is devoted to showing a different control on the geometry of the Markov partition, namely that the set satisfies a generalisation of the notion of  a John domain. This piece of information is used in order to have a convenient interpolation result (Lemma~\ref{lem:interpol}).
 Note that the construction of Bowen~\cite{Bowen73} and Ratner~\cite{Ratner73} are very similar but that Bowen's later description~\cite{Bowen75} of the construction of Markov partitions is described rather differently. The later method of construction is based on shadowing in a way that works elegantly for all Axiom~A systems. However the geometry is rather lost in the construction and a clear hold of the geometry is precisely what we require for our present purposes. In this appendix we will follow the construction of Ratner~\cite{Ratner73} and for clarity use, whenever possible, identical notation as used in this reference.

Throughout this section we assume the setting of a transitive Anosov flow \(\flow{t}:\cM \to\cM\).
First we recall the notation and the general idea behind the construction of the Markov partition. 
For any \(x\) let \(W^s_\epsilon(x)\) (resp.\@ \(W^{cs}_\epsilon(x)\),\(W^u_\epsilon(x)\),\(W^{cu}_\epsilon(x)\)) 
denote the \(\epsilon\)-sized local stable (resp.\@ centre-stable, unstable, centre-unstable) manifolds centred at \(x\). 
As usual, we know that there exists \(\epsilon_0,\gamma>0\) such that, for all \(x\) and for all 
\(y\in W^s_\gamma(x)\),  \(z\in W^{cu}_\gamma(x)\) the sets \(W^s_{\epsilon_0}(x)\) and \(W^{cu}_{\epsilon_0}(x)\) 
intersect in exactly one point which we denote by \([y,x]\), this defines the well known \textit{canonical coordinates}.
From now on we suppose that such a choice of \(\epsilon_0,\gamma>0\) is fixed.
Let \(\C \subset W^u_\gamma(x)\), \(\D \subset W^s_\gamma(x)\). A \emph{parallelogram} is a set \(\A=[\C,\D]\) defined as all the points 
\([y,z]\) such that \(y\in \C\), \(z\in \D\). Observe that the set \(\A\) is foliated by stable manifolds but, 
in general, will not be foliated by unstable manifolds. 
Let \(\mathfrak{A} = \{ \A_1,\ldots\,\A_k \}\), \(\A_i = [\C_i, \D_i]\), \(\A_i \cap \A_j = \emptyset\) for \(i\neq j\), be a finite complete system of parallelograms. (Here \emph{complete} means that for every point in \(\cM\) there is an interval on the trajectory of the point whose end points each lie in one of the parallelograms.) 
Let \(\M_{\mathfrak{A}}\) be the set theoretic union of the parallelograms \({\{\A_i\}}_{i}\) with the induced topology.
Let \(\ell(x)\), \(x\in \M_{\mathfrak{A}}\), denote the length of the interval of the trajectory of the flow \(\flow{t}\) extending from \(x\) to its first intersection \(x'\) with \(\M_{\mathfrak{A}}\). Let \(T\) denote the one-to-one mapping of \(\M_{\mathfrak{A}}\) onto itself which maps \(x\) to \(x'\).
 A system \(\mathfrak{A}\) is said to be Markovian  for the flow \(\flow{t} \) if, whenever \(x\in \Int \A_i \cap T^{-1} (\Int \A_j) \),\footnote{As usual \(\D_i(x)\) denotes the \(\D_i\) such that \(x\in \A_i = [\C_i,\D_i]\). Similarly for \(\C_i(x)\).}
 \begin{equation}\label{eq:Markov} 
      T(\Int \D_i(x)) \subset \D_j(T(x)) \quad \text{and} \quad
      T(\C_i(x)) \supset \Int \C_j(T(x)).
 \end{equation}
 
As mentioned previously we rely on the following result.
\begin{thm}[{\cite[Theorem 2.5]{Bowen73} or \cite[Theorem 2.1]{Ratner73}}]
 For every \(\epsilon>0\) the transitive Anosov flow \(\flow{t}:\cM \to \cM\) has a Markov partition with the size of the elements of the partition being at most \(\epsilon\).
\end{thm}
Since we will need more details of the construction of the Markov partition, particularly some information on the geometry of the partition elements we here recall the most relevant details of the construction.
During the construction \(\alpha,\delta>0\) are chosen to satisfy, amongst other conditions, the requirement that  \(0 < \alpha < \delta< \min(\epsilon,\gamma,\epsilon_0)\).
To start the construction we fix \(\mathfrak{A}^0 = \{ \A_1^0,\ldots\,\A_k^0 \}\), a complete finite system of parallelograms \(\A_i^0 = [\C_i^0, \D_i^0]\), \(\C_i^0 = W_u^\alpha(x_i)\), \(\D_i^0 = W_s^\alpha(x_i)\). 
By a recursive procedure \cite[\S2]{Ratner73} we define the sets \(\C_{i}^{n} \subset W_{u}^{\delta}(x_{i})\) and  \(\D_{i}^{n} \subset W_{s}^{\delta}(x_{i})\). This procedure involves applying a strong contraction to the sets already defined in order to add small additional sets to the sets already defined and consequently become closer to being Markov. 
At the beginning some \(m\) is chosen sufficiently large. For each \(i,j\) we consider if \(\flow{-m}\C_j^n\) contributes a part which should be added to the set \(\C_i\). 
  The successive approximation means that these leaves converge to the Markov property.
The unstable part of the partition element is defined by a countable union
\[
  \C_i=\overline{\bigcup_{n \geq 1}  \C_{i}^{n}} \subset W_\delta^u(x_i).
\]
The stable part, \(\D_{i}\), is defined similarly but using \(\flow{m}\) in place of \(\flow{-m}\).

\subsection{Box-counting Dimension of the Boundary}
\label{app:DimOfBoundary}
The structure of the constructed Markov partition leads to the following result.
\begin{prop}[{\cite[Proposition 3.5]{HV05}}]
 \label{prop:dimbound}
 The box-counting dimension of the union of the unstable boundaries of the elements of the Markov partition of an Anosov map is strictly smaller than the dimension of the unstable bundle. 
\end{prop}
\noindent
The proof of the above is based on estimates available in Bowen~\cite{Bowen75} and a standard relation~\cite{Falconer03} which connects the measure of a neighbourhood of a set to the box-counting dimension of that set. 
Although the result stated is for Anosov diffeomorphisms the same result holds without issue for the Markov structure of an Anosov flow as described above.

We will use this information about the box-counting dimension of the boundary to prove the previously stated Lemma~\ref{lem:PartOfUnity} which concerns the existence of a partition of unity. This construction is essentially standard but since the details are crucial and the estimates concerning the boundary of the set are less common, we give here the details of the construction and the proof of the required estimates.

Fix a function \(\Phi \in \cC^1(\bR,[0,1])\) such that \(\Phi(u) = 1\) whenever \(\abs{u} \leq \frac{1}{4}\),  that  \(\Phi(u) = 0\) whenever \(\abs{u} \geq \frac{3}{4}\) and \(\sum_{k=-\infty}^{\infty} \Phi(u-k) = 1\) for all \(u\in \bR\). 
(For any \(x\in \bR^{d}\), \(r>0\) we denote by \(\ball{x}{r}\) the ball which is centred at \(x\) and has radius \(r>0\).)
For each \(\epsilon>0\), \(\ell = (\ell_1,\ldots, \ell_d) \in \bZ^{d}\) define \( \Phi_{\ell}^{(\epsilon)} \in \cC^1(\bR^d,[0,1])\) by 
\[
 \Phi_{\ell}^{(\epsilon)}(x_1,\ldots,x_d)
 :=
 \prod_{k=1}^{d} \Phi\left(\epsilon^{-1}(x_k - \epsilon \ell_k)\right).
\]
Such a function is ``centred'' at the point \(\epsilon \ell = (\epsilon \ell_1,\ldots,\epsilon \ell_d)\in \bR^d\).
Observe that 
\begin{itemize}
\item
 The support of \(\Phi_{\ell}^{(\epsilon)}\) 
 is contained within
 \(  \ball{\epsilon \ell}{ \frac{3\epsilon}{4}} \),
\item
For all \(x \in \ball{\epsilon \ell}{\frac{\epsilon}{4}}\)
\[
 \Phi_{\ell}^{(\epsilon)}(x) = 1,
\]
\item
For each \(x \in \bR^d\),
\[
 \sum_{\ell \in \bZ^{d}}
 \Phi_{\ell}^{(\epsilon)}(x) = 1,
 \]
 \item
 There exists some \(K>0\) such that \(\norm{ \smash{ \Phi_{\ell}^{(\epsilon)} }}_{\cC^1} \leq K \epsilon^{-1}\) for all \(\epsilon >0\), \(\ell\in \bZ^{d}\).
\end{itemize}

We suppose that \(\Omega \subset \bR^{d}\) is bounded and that \(\partial \Omega\) has box-counting dimension strictly less than \(d\).
That the set is bounded means there exists \(K>0\) such that the cardinality of the set 
\(
 \{  \ell \in \bZ^{d}  : \ball{\epsilon \ell}{\frac{3\epsilon}{4}} \cap \Omega \neq \emptyset  \}
\)
is bounded from above by \(K \epsilon^{-d}\).

Consider the \(\epsilon\)-mesh where the cubes of the mesh are centred on the points \(\{ \epsilon \ell : \ell \in \bZ^{d}\}\).
For any set \(E\subset \bR^d\) let \(N_\epsilon(E)\) denote the number of cubes in the \(\epsilon\)-mesh which intersect \(E\). (There are several equivalent definitions of box-counting dimension \cite[\S3.1]{Falconer03}.)
Since the boundary \(\partial \Omega\) had box-counting dimension strictly less than \(d\), we know that there exists \(K>0\), \(\cdim \in [0, d)\) such that, for all \(\epsilon >0\),
\[
 N_{\epsilon}(\partial \Omega) \leq K \epsilon^{-\cdim}.
\]
Consequently the cardinality of the set 
\(\{  \ell \in \bZ^{d} : \ball{\epsilon \ell}{\frac{3\epsilon}{4}} \cap \partial\Omega \neq \emptyset  \}
\)
is bounded from above by \(K \epsilon^{-\cdim}\) (increasing \(K>0\) if required, independently of \(\epsilon\)).
This completes the proof of Lemma~\ref{lem:PartOfUnity}.

\subsection{Markov partitions are almost John}\label{sec:john}

The construction of the unstable part of the Markov partition can be conveniently rephrased as follows (for full details consult \cite{Ratner73}). 
There is a collection of sets  \({\{ \C_i \}}_{i=1}^{N}\) where each \(\C_{i}\) is a bounded subset of \(\bR^d\). For each set there is a subset \(\C_i^0 \subset \C_i \) which has nice geometry in the sense that the boundary of \(\C_i^0\) is \(\cC^1\). 
Let \(\mathfrak{C} \) denote the disjoin union \( \bigsqcup_{i} \C_i \).
There is a map \(T :  \mathfrak{C} \to \mathfrak{C}\) 
which corresponds to the Anosov flow for some large time (after projecting along local stable manifolds). 
There is an index set \(\cA \subset \{1,\ldots,N\}^2\) and, for each \((j,k) \in \cA\), a map \(h_{j,k}: \C_j \to \C_k\) such that \(T \circ h_{j,k} = \id\). Moreover these maps are strong contractions in the sense that there exist \(0 < \lambda_2 \leq \lambda_1 < 1\) such that, for all \((j,k) \in \cA\) and \(x,y  \in \C_{j}\),
\[
 \lambda_2 \dist{x}{y} \leq 
 \dist{h_{j,k}(x)}{h_{j,k}(y)} \leq 
 \lambda_1 \dist{x}{y}.
\]
Define
\[
 \C_j^n = \bigcup_{k: (j,k) \in \cA} h_{j,k}\left( C_k^{n-1}\right),
\quad \quad
 \C_i=\overline{\bigcup_{n \geq 1}  \C_{i}^{n}}.
\]
Note that \(\C^{n}_{i}\supset \C^{n-1}_{i}\) for all \(n\).
That the sets have this above structure suffices to show some modest control on the geometry.

 Since  \(0 < \lambda_2 \leq \lambda_1 < 1\) there exists \(s\geq1\) such that
  \begin{equation}
   \label{eq:defs}
     \lambda_{2} = \lambda_{1}^{s}.
  \end{equation}
Observe that \(s\geq 1\) can be taken to be equal to \(1\) in the special case when the expansion is isotropic. This is the situation in the special case when the unstable bundle is one-dimensional.
Recall (Definition~\ref{def:john}) that a set  \(\Omega \subset \bR^d\)  is \emph{almost John} in the sense that there exist \(\kFill,\epsilon_0>0\) such that, for all \(\epsilon \in (0,\epsilon_0)\) and for all \(x\in \Omega\), there exists \(y\in \Omega\) such that \(\dist{x}{y} \leq  \epsilon\) and  \(\ball{y}{\kFill \epsilon^s} \subset \Omega\).

\begin{lem}
\label{lem:fillball}
Each set  \(\C_i\)  is \emph{almost John}. The exponent \(s\geq 1\) is that given by \eqref{eq:defs}. 
\end{lem}

\begin{proof}
Let \(\delta>0\) be such that \(\operatorname{diam}(\C_i^0) \leq \delta\) for each \(i\). Since the set \(\C_i^0\) has smooth boundary there exists \(\kZero>0\) and \(\epsilon_1 > 0\) such that:
  For all \(i\),  for all \(x \in \C_i^0\) and for all \(\epsilon \in (0,\epsilon_1)\) there exists \(y\in  \C_i^0\) such that 
  \[
  \ball{y}{\kZero\epsilon} \subset \C_i^0, \quad \text{ and } \quad \dist{x}{y} \leq  \epsilon.
  \]
 Fix the constants
 \[
 \kThree = \tfrac{2\delta}{\lambda_{1}(1-\lambda_{1})},
 \quad \quad
 \kFour = \min(  \tfrac{1}{2}, \epsilon_{0}\kThree^{-1} ).
 \]
 Let \(\epsilon \in (0,\epsilon_0)\). Define \(N_\epsilon \in \bN\) by the requirement that 
 \( \kThree \lambda_1^{N_\epsilon+1} \leq \epsilon \leq \kThree \lambda_1^{N_\epsilon} \). 
 For \(x\in \C_i \), we will consider two cases.
 
 \underline{\textbf{Case 1}} (\(x \in \C_i^{N_\epsilon}\))\textbf{:}
 Let \(j\) be such that \(T^n x \in \C_j\).
 We know that there exists \(y \in \C_i^{N_\epsilon} \) such that \(T^{N_\epsilon} y \in \C_j^0\), 
 \(\dist{T^{N_\epsilon} x}{T^{N_\epsilon} y} \leq  \epsilon'\) 
 and \(B(T^{N_\epsilon} y,  \kZero  \epsilon') \subset \C_i^0\). 
 Consequently \(\dist{x}{y} \leq \epsilon'  \lambda_1^{{N_\epsilon}} \)
  and  \(T^{{N_\epsilon}} B(y, \kZero \epsilon'  \lambda_{2}^{{N_\epsilon}}) \subset  B(T^{N_\epsilon} y, \kZero \epsilon') \subset \C_i^0\).
 We choose \(\epsilon' = \epsilon  \kFour \lambda_1^{-N_{\epsilon}} \in (0,\epsilon_0)\). 
 This means that  
 \[
 \dist{x}{y} \leq \epsilon  \kFour  \leq \epsilon
  \]
  as required.
 Using also that the definition of \(s>1\) implies   
\( \lambda_{2}/\lambda_{1} = \lambda_{1}^{s-1}\) we see that
 \[
 \kZero \epsilon'  \lambda_{2}^{{N_\epsilon}} 
 =\kZero \kFour   (\tfrac{\lambda_{2}}{\lambda_1})^{N_{\epsilon}} \epsilon
 =\kZero \kFour  {\lambda_1}^{N_{\epsilon}(s-1)} \epsilon
 \geq \kZero \kFour  (\tfrac{\epsilon}{\kThree})^{s-1} \epsilon
 = \tfrac{\kZero \kFour}{\kThree^{s-1}} \epsilon^{s}. 
 \]
 This means that we have shown that \(\ball{y}{\kZero' \epsilon^{s}   }  \subset \C_i^{N_{\epsilon}}\) where \(\kZero' =\tfrac{\kZero \kFour}{\kThree^{s-1}}\).
 
  \underline{\textbf{Case 2}} (\(x \in \C_i \setminus \C_i^{N_\epsilon}\))\textbf{:}
 In this case we know that there exists \(z\in \C_i^{N_\epsilon}\) such that 
 \(\dist{x}{z} \leq \delta \frac{\lambda_{1}^{N_{\epsilon}}}{1 - \lambda_{1}}\). This is because, from the construction, the diameter of every component of \(\C_i^n\) is not greater than \(\delta \lambda_1^{n} \) and must intersect some previously defined set. 
 Using now what we demonstrated in the other case we know that there exists some 
 \(y \in  \C_i^{N_\epsilon} \) which satisfies
 \(\dist{z}{y} \leq \kFour \epsilon\) and \(\ball{y}{\kZero' \epsilon^{s}}) \subset \C_{i}^{N_{\epsilon}}\).
 Observe that 
 \[
 \dist{x}{y} 
 \leq \delta \tfrac{\lambda_{1}^{N_{\epsilon}}}{1 - \lambda_{1}} + \kFour \epsilon 
 \leq \left(    \tfrac{\delta}{\lambda_{1} (1 - \lambda_{1}) \kThree} + \tfrac{1}{2} \right) \epsilon 
  \leq \epsilon
 \]
  as required.
\end{proof}

\begin{rem}
The work of Avila, Gou\"ezel \& Yoccoz~\cite{AGY06} required the domain of the expanding Markov map to be a John domain in a sense which corresponds to our definition if \(s=1\). However, when the expansion is not the same in all directions, it seems unlikely that a condition better than we use here could be satisfied.
A weakening of the definition of a John domain in a similar way as we use has been studied in other contexts (see, e.g., \cite{Nieminen06} and references within). 
 In the case \(s=1\) the John domain property implies~\cite[Corollary 6.2]{Nieminen06}  the estimate on the box-counting dimension of the boundary. However, in general when \(s>1\), this is not sufficient~\cite[\S7.3]{Nieminen06} for a useful estimate of the dimension. 
 We therefore show independently the two properties which we require.
\end{rem}

In our application we use the above lemma for the following key interpolation result. 

\begin{lem}
 \label{lem:interpol}
 Let \(\Omega \subset \bR^d\) be almost John with exponent \(s\geq 1\). Let \(\gamma = 1/s \in (0, 1]\).
 There exists \(\kInter>0\) and \(\epsilon_1 >0\) such that, for all \(\epsilon \in (0,\epsilon_1)\) and \(f\in \cC^\alpha(\Omega)\),
 \[
  \norm{f}_{L^\infty(\Omega)} \leq \kInter\left( \epsilon^{-d}   \norm{f}_{L^1(\Omega)}  +  \epsilon^\gamma \abs{f}_{\cC^{\alpha}(\Omega)} \right).
 \]
\end{lem}
\begin{proof}
 Let \(x \in \Omega\), \(\epsilon \in (0,\epsilon_1)\) and  \(f\in \cC^\alpha(\Omega)\). 
 Since \(\Omega\) is almost John there exists some \(y\in \Omega\) such that \(\ball{y}{\kFill \epsilon}\subset \Omega\) and \(\dist{x}{y} \leq  \epsilon^{\gamma}\). 
 The volume of the \(d\)-ball of radius \(\epsilon\) is equal to  \(K_{d} \epsilon^d\)
 and consequently there must exist \(z \in \ball{y}{\epsilon}\) such that \(\abs{f(z)} \leq K_{d}^{-1} \epsilon^{-d} \norm{f}_{L^1(\Omega)}\) because otherwise there would be a contradiction for the \(L^1\) norm   (if the statement were false then \(\abs{f(z)} > K_{d}^{-1} \epsilon^{-d} \norm{f}_{L^1}(\Omega)\)  for all \(z \in \ball{y}{\epsilon}\) and consequently \(\norm{f}_{L^1(  \ball{y}{\epsilon} )} > \norm{f}_{L^1(\Omega)}   \)).
 This means that 
 \[
 \begin{aligned}
  \abs{f(x)} 
  &\leq \abs{f(z)} + \abs{f(x)-f(z)}\\
  &\leq 
  K_{d}^{-1} \epsilon^{-d} \norm{f}_{L^1(\Omega)}
  + \abs{f}_{\cC^{\alpha}(\Omega)} \dist{x}{z}^\alpha\\
  &\leq
  \kInter\left( \epsilon^{-d}   \norm{f}_{L^1(\Omega)}  +  \epsilon^\gamma \abs{f}_{\cC^{\alpha}(\Omega)} \right).
 \end{aligned}
 \]
This estimate holds for all \(x \in \Omega\),  \(\epsilon \in (0,\epsilon_1)\) and  \(f\in \cC^\alpha(\Omega)\). 
\end{proof}

\end{document}